\documentclass [11pt,oneside]{amsart}
\usepackage{xcolor}
\usepackage{amssymb,amsmath,amsthm,amsfonts}
\usepackage{tikz-cd}
\usepackage{enumerate}
\usepackage{setspace}
\usepackage{empheq}
\usepackage[all]{xy}
\usepackage{verbatim}
\usepackage[normalem]{ulem}
\usepackage{todonotes}
\usepackage[bookmarks,colorlinks,breaklinks]{hyperref}  
\hypersetup{linkcolor=blue,citecolor=blue,filecolor=blue,urlcolor=blue} 

\numberwithin{equation}{section}
\numberwithin{figure}{section}

\newtheorem{theorem}{Theorem}[section]
\newtheorem{proposition}[theorem]{Proposition}

\newtheorem{lemma}[theorem]{Lemma}

\newtheorem{corollary}[theorem]{Corollary}
\newtheorem{conjecture}[theorem]{Conjecture}

\theoremstyle{definition}
\newtheorem{definition}[theorem]{Definition}

\newtheorem{remark}[theorem]{Remark}

\definecolor{myblue}{rgb}{0.6, 0.9, 1}

\makeatletter

\newcommand{\Rmnum}[1]{\expandafter\@slowromancap\romannumeral #1@}
\makeatother

\definecolor{myblue}{rgb}{0.6, 0.9, 1}
\definecolor{mygreen}{rgb}{0,0,1}
\definecolor{purple}{rgb}{0.6,0.2,1}
\definecolor{orange}{rgb}{0.8,0,0.2}

\newcommand{\bC}{\mathbb{C}}
\newcommand{\bP}{\mathbb{P}}
\newcommand{\C}{\mathbb{C}}
\newcommand{\Q}{\mathbb{Q}}

\newcommand{\bZ}{\mathbb{Z}}

\newcommand{\bQ}{\mathbb{Q}}
\newcommand{\bR}{\mathbb{R}}

\newcommand{\bN}{\mathbb{N}}

\newcommand{\la}{\lambda}

\newcommand{\al}{\alpha}

\newcommand{\Gal}{\operatorname{Gal}}
\newcommand{\Kbar}{\overline{K}}
\newcommand{\Qbar}{\overline{\bQ}}

\newcommand{\<}{\langle}
\renewcommand{\>}{\rangle}

\newcommand\iso{\simeq}
\newcommand{\bbf}{\mathbf{f}}
\newcommand{\bbpsi}{\mathbf{\Psi}}
\newcommand{\bbx}{\mathbf{x}}
\newcommand{\bbphi}{\mathbf{\Phi}}
\newcommand{\bc}{\mathbf{C}}

\newcommand{\bbc}{\mathbf{c}}

\newcommand{\di}{\mathrm{div}}
\author{Niki Myrto Mavraki and Harry Schmidt}
	\email{myrto.mavraki@utoronto.ca \\ Department of Mathematical and Computational science \\ University of Toronto Mississauga\\ ON L5L 1C6, Canada }
	\email{Harry.Schmidt@warwick.ac.uk \\ Department of Mathematics\\ University of Warwick \\ CV4 7AL, Coventry}


\begin{document}
	\title[Dynamical Bogomolov conjecture in families]{On the dynamical Bogomolov conjecture for families of split rational maps}
	
	\date{\today}
	
	\begin{abstract}
		We prove that Zhang's dynamical Bogomolov conjecture holds uniformly along $1$-parameter families of rational split maps and curves.  
         This provides dynamical analogues of recent results of Dimitrov--Gao--Habegger \cite{DGH:pencils,DGH:uml} and K\"uhne \cite{kuhne:uml}. 
          In fact, we prove a stronger Bogomolov-type result valid for families of split maps in the spirit of the  relative Bogomolov conjecture \cite[Conjecture 1.2]{DGH:assumerbc}. We thus provide first instances of a generalization of Baker-DeMarco's conjecture \cite[Conjecture 1.10]{BD:polyPCF} to higher dimensions. 
         Our proof contains both arithmetic and analytic ingredients. Our main analytic result may be viewed as a dynamical  Ax--Lindemann-type theorem for split rational endomorphisms. More precisely, we show that weakly special curves under 
         the action of a split map $(f,g)$ of $(\bP_{\bC}^1)^2$ are exactly those that lead to linear relations between the measures of maximal entropy of $f$ and $g$. This extends a previous result of Levin-Przytycki \cite{LP}. 
         We further establish a height inequality for families of split maps and varieties comparing the values of a fiber-wise Call-Silverman canonical height with a height on the base and valid for most points of a non-preperiodic variety. This provides a dynamical generalization of \cite[Theorem 1.3]{Habegger:special} and generalizes results of Call-Silverman \cite{Call:Silverman} and Baker \cite{Baker:finiteness} to higher dimensions. In particular, we establish a geometric Bogomolov theorem for split rational maps and varieties of arbitrary dimension. 
	\end{abstract}
	\maketitle
	\section{Introduction}

	Motivated by the Manin-Mumford conjecture, proved in \cite{Raynaud:1}, and its strengthening as in Bogomolov's conjecture, proved in \cite{Ullmo:Bogomolov, ZhangBog}, Zhang conjectured that analogous statements  \cite[Conjectures 1.2.1 and 4.1.7]{Zhang:distributions} hold in a more general dynamical setting. 
    Our goal in this paper is to prove uniform and relative versions of the dynamical Bogomolov conjecture for $1$-parameter families of split rational maps.

    Our main results are stated as Theorem \ref{umb}, Theorem  \ref{relative bogomolov}, Theorem \ref{transcendence} and Theorem \ref{splitineq} in what follows. Our efforts culminate in the dynamical relative Bogomolov Theorem \ref{relative bogomolov}, which implies the uniformity in the dynamical Bogomolov-type Theorem \ref{umb} and whose proof uses Theorems \ref{transcendence} and \ref{splitineq}.
    To ease the reader into our work, we begin with an application that may be of independent interest in complex dynamics. 
    Throughout $\mathrm{Prep}(\phi)$ denotes the set of preperiodic points of a self-map of a set $\phi$ and we let $\mathrm{Rat}_d$ be the space of degree $d$ rational functions. 
    A first conjecture  \cite[Conjecture 1.4]{DKY:uni} hinting at uniformity in the dynamical Manin-Mumford conjecture is as follows. 
    
    \begin{conjecture}[DeMarco-Krieger-Ye]\label{diagonal}
    Let $d\ge 2$. There exists $M>0$ such that for any two rational functions $f, g\in \mathrm{Rat}_d(\bC)$ we have 
    $$\text{either }\mathrm{Prep}(f)=\mathrm{Prep}(g)\text{ or }\#\mathrm{Prep}(f)\cap \mathrm{Prep}(g)<M.$$
    \end{conjecture}
    
     As evidence, DeMarco, Krieger and Ye show \cite[Theorem 1.1]{DKY:uni} that Conjecture \ref{diagonal} holds along the surface in $\mathrm{Rat}_d\times \mathrm{Rat}_d$ given by the unicritical polynomials. 
           \begin{theorem}[DeMarco-Krieger-Ye]
           There is a constant $M>0$ such that for $t_1,t_2\in\bC$
           $$\text{either }t_1=t_2\text{ or }\#\mathrm{Prep}(z^2+t_1)\cap \mathrm{Prep}(z^2+t_2)<M.$$
           \end{theorem}
          It is well known \cite{Beardon} that $t_1= t_2$ if and only if $\mathrm{Prep}(z^2+t_1)= \mathrm{Prep}(z^2+t_2)$ if and only if the Julia sets of $z^2+t_1$ and $z^2+t_2$ agree.

     A first application of our results shows that  Conjecture \ref{diagonal} holds along arbitrary curves in $\mathrm{Rat}_d\times \mathrm{Rat}_d$ defined over the algebraic numbers. 
   
              \begin{theorem} \label{moduli:curve}
                  Let $C\subset \mathrm{Rat}_d\times \mathrm{Rat}_d$ be a curve defined over $\Qbar$. 
                                Then, there exists $M = M(C)>0$ such that for all $(f,g)\in C(\bC)$ we have 
                                	$$\text{either }\mathrm{Prep}(f)=\mathrm{Prep}(g)\text{ or }\#\mathrm{Prep}(f)\cap \mathrm{Prep}(g)<M.$$
              \end{theorem}

      Note here that the common preperiodic points of $f$ and $g$ in Conjecture \ref{diagonal} can be seen as the preperiodic points under the action of the split map $(f,g):\bP_1^2\to \bP_1^2$, with $(f,g)(x,y)=(f(x),g(y))$, which lie on the diagonal in $\bP_1^2:=\bP^1\times \bP^1$. 
      Replacing the diagonal with a different irreducible curve $C$, Ghioca, Nguyen and Ye \cite{dmm1} have established that 
      \begin{align}
      \#\mathrm{Prep}(f,g)\cap C(\bC)<+\infty, 
      \end{align}
      unless $C$ is preperiodic under the action of $(f,g)$; that is if $f$ or $g$ is non-exceptional. They thus established Zhang's dynamical Manin-Mumford conjecture for split maps and curves. 
      Note here that the assumption that the maps $f$ and $g$ are non-exceptional ensures that counterexamples to Zhang's original conjectures are avoided; see \cite{mm:counter, mm:new}.
      
      	Also in view of recent breakthroughs of Dimitrov, Gao and Habegger \cite{DGH:uml} and K\"uhne \cite{kuhne:uml}, who established a uniform version of the classical Manin-Mumford and Bogomolov conjectures, one may wonder when is $\#\mathrm{Prep}(f,g)\cap C(\bC)$ bounded uniformly. 
          In Theorem \ref{umb} we show that this holds if the maps and the curve vary along a $1$-parameter family and are defined over the algebraic numbers.
          
      	K\"uhne has further established a stronger relative version of Bogomolov's conjecture \cite[Conjecture 1.2]{DGH:assumerbc} for varieties in fibered products of elliptic families \cite{kuhne:relative}. 
      	In Theorem \ref{relative bogomolov} we prove instances of an analogous statement in the setting of families of split maps.

   \subsection{Setup}\label{setup}
      In order to state our results we introduce some notation and terminology. 
	Throughout we let $B$ be a projective, regular, irreducible curve over $\Qbar$ and write $K=\Qbar(B)$. We often pass to finite extensions of $K$ and by an irreducible variety over $K$ we always mean irreducible over $\overline{K}$.
	We let $\bbphi:\mathbb{P}^{\ell}_1\to \mathbb{P}^{\ell}_1$ be an endomorphism defined over $K$ for $\ell\in \bN_{\ge 1}$. 
	We say that $\bbphi$ is a \emph{split polarized endomorphism} of polarization degree $d\ge 2$ over $K$ if it is defined as $\bbphi(z_1,\ldots,z_{\ell})=(\bbf_1(z_1),\ldots,\bbf_{\ell}(z_{\ell}))$, where $\bbf_1,\ldots,\bbf_{\ell}$ are rational functions of fixed degree $d\ge 2$ defined over $K$. Recall that $\bbf\in K(z)$ is isotrivial if there is a M\"obius transformation $\mathbf{M}\in \overline{K}(z)$ such that $\mathbf{M}\circ \bbf \circ \mathbf{M}^{-1}\in \Qbar(z)$. If all $\bbf_i$ are non-isotrivial for all $i=1,\ldots,\ell$, then we say that $\bbphi$ is \emph{isotrivial-free}. 
	The map $\bbphi$ defines an endomorphism of a quasiprojective variety $\Phi: B_0\times \mathbb{P}^{\ell}_1\to B_0\times \mathbb{P}^{\ell}_1$, where $B_0$ is a Zariski open and dense subset of $B$ and $\Phi$ is defined over $\Qbar$. 
	Similarly a (geometrically) irreducible subvariety $\mathbf{X}\subset \mathbb{P}^{\ell}_1$ over $K$ gives a subvariety $\mathcal{X}\subset B_0\times\mathbb{P}^{\ell}_1$ over $\Qbar$.
    
	We denote by $\pi: B_0\times\mathbb{P}^{\ell}_1\to B_0$ the projection map and for each $t\in B_0(\bC)$ we write $\Phi_t ={\Phi}|_{\pi^{-1}(\{t\})}$ and $\mathcal{X}_t=\mathcal{X}\cap \pi^{-1}(\{t\})$.
	We then have a fiber-wise Call-Silverman canonical height \cite{Call:Silverman} associated to $\Phi$, assigning to each point $Z:=(t,z_1,\ldots,z_{\ell})\in (B_0\times\mathbb{P}^{\ell}_1)(\Qbar)$ the value of a canonical height in its fiber as
    \begin{align}\label{fiberwiseheight}
    \hat{h}_{\Phi}(Z)=\hat{h}_{\Phi_{\pi(Z)}}(Z)=\sum_{i=1}^{\ell}\hat{h}_{f_{i,t}}(z_i)\in \mathbb{R}_{\ge 0}.
    \end{align}
    We note here that the canonical height $\hat{h}_{\Phi}(Z)$ vanishes exactly when $Z\in \mathrm{Prep}(\Phi_{\pi(Z)})$. 
    More generally it can be thought of as a measure of distance between $Z$ and the preperiodic points of the specialization $\Phi_{\pi(Z)}$ on its fiber.

 \emph{ Throughout this article, let $B_0$ be an unfixed sufficiently small
open and dense subset of $B$. In particular, we can assume that
$\Phi_t$ is a polarized endomorphism of polarization degree $d\ge 2$ and that $\mathcal{X}_t$ is irreducible for all $t\in B_0(\Qbar)$.}

  \subsection{Instances of uniformity in the dynamical Bogomolov conjecture}
  Our first main result establishes the first instance of uniformity in the dynamical Bogomolov and Manin-Mumford conjectures \cite[Conjectures 1.2.1 and 4.1.7]{Zhang:distributions} or \cite[Conjecture 1.2]{mm:new} for varying non-polynomial maps or along a family of varying curves. 
     For the statement to hold, one requires the notion of a weakly special curve (see Definition \ref{weakly special}), as we shall explain shortly. 
     
    	\begin{theorem}\label{umb}
    		Let $\bbphi:\bP^2_1\to \bP^2_1$ be a split polarized endomorphism over $K$ with polarization degree $\ge 2$.  Let $\bc\subset \bP^2_1$ be an irreducible curve defined over $K$ that is not weakly $\bbphi$-special. 
    		 Then there exist $\epsilon>0$ and $M>0$ depending only on $\bbphi$ and $\bc$ such that
    		\begin{align}\label{setfinite}
    			\#\{z\in \mathcal{C}_t(\Qbar)~:~\hat{h}_{\Phi_t}(z)<\epsilon \}\le M,
    		\end{align}
    		for all but finitely many $t\in B_0(\Qbar)$. 
            In particular, 
            \begin{align*}
              		\#\mathcal{C}_t(\Qbar)\cap \mathrm{Prep}(\Phi_t)\le M,
              		\end{align*}
                for all but finitely many $t\in B_0(\Qbar)$.    
    	\end{theorem}
    	
        That one should exclude certain $t\in (B\setminus B_0)(\Qbar)$ in the conclusion of the theorem is clear. For  instance if $\mathcal{C}_t$ is $\Phi_t$-preperiodic, then it contains infinitely many $\mathrm{Prep}(\Phi_t)$-preperiodic points. 
        That there are at most finitely many such exceptional $t$ is not a priori clear but follows from our theorem. 
        More generally, we can show that for all but finitely many $t$ the curve $\mathcal{C}_t$ stays a distance away from being $\Phi_t$-preperiodic. The said distance is measured by a real valued Arakelov-Zhang pairing function 
                       $$t\mapsto \hat{h}_{\Phi_t}(\mathcal{C}_t);$$
        taking non-negative values for $t\in B_0(\Qbar)$ (see \S\ref{background}.) 
        In view of Zhang's fundamental inequality \cite{Zhang:positivesurfaces}, characterizing its zeros amounts to the dynamical Bogomolov conjecture \cite[Theorem 1.1]{dmm1}.
              More precisely, as a consequence of Theorem \ref{umb}, we obtain the following.
                   
                   \begin{corollary}\label{small curves}
                   Let $\bbphi:\bP^2_1\to \bP^2_1$ be a split polarized endomorphism over $K$ with polarization degree $\ge 2$. Let $\bc\subset \bP^2_1$ be an irreducible curve defined over $K$ that is not weakly $\bbphi$-special. Then the following hold. 
                   \begin{enumerate}
                   \item 
                   There are at most finitely many $t\in B_0(\bC)$ such that $\mathcal{C}_t$ is $\Phi_t$-preperiodic.
                   \item 
                   There is $\epsilon>0$ such that the set 
                              $\{t\in B_0(\Qbar)~:~\hat{h}_{\Phi_t}(\mathcal{C}_t)<\epsilon\}$
                    is finite.           
                   \end{enumerate}
                   \end{corollary} 
        
        Let us now explain the necessity of a condition on $\bc$ in the statement of Theorem \ref{umb}. 
        First, recalling that preperiodic  curves contain infinitely many preperiodic points, we need to exclude the possibility that $\bc$ is $\bbphi$-preperiodic for the conclusion of Theorem \ref{umb} to remain true. 
        If $\bbphi$ is given by a pair of Latt\`es maps, which are associated to algebraic groups, then we further have to assume that $\bc$ is not the projection of a curve with genus $1$ as an elliptic curve contains infinitely many points of finite order. 
         Moreover, if the curve is vertical $\bc=\bbc\times \bP_1$ for a point $\bbc\in K$ assuming that it is not $\bbphi$-preperiodic does not suffice for the conclusion to hold. Indeed, by an application of Montel's theorem there might still be infinitely many $t\in B_0(\Qbar)$ such that $\mathcal{C}_t$ is $\Phi_t$-preperiodic; see \cite{DF, D:stableheight}. 
              To make matters worse, even excluding all such $t$ one cannot hope to find uniform $\epsilon,M>0$ satisfying \eqref{setfinite}. If for instance $\bbf_1$ is not isotrivial, one can solve the  equations $f_{1,t}^n(c(t))= 1$ for varying $n$ and apply  \cite[Theorem 1.4]{Call:Silverman}. 
            
            Our notion of a weakly $\bbphi$-special curve aims to summarize all such necessary assumptions. 
            As our examples reveal, we need to consider maps that are associated to groups separately.             	
             To this end, we recall here the notion of an exceptional map. 
   \begin{definition} \label{ordinary} For this definition we assume that all objects are defined over some algebraically closed field with characteristic zero. We say that a rational map $f: \mathbb{P}_1\rightarrow \mathbb{P}_1$ of degree $\deg(f) \geq 2$ is \emph{exceptional} if there exists an algebraic group $G_f$ (which is either an elliptic curve or the multiplicative group), an isogeny $\varphi_f: G_f\rightarrow G_f$ and a dominant map $\Psi_f: G_f \rightarrow \mathbb{P}_1$ such that 
   	$$\Psi_f\circ \varphi_f= f\circ \Psi_f.$$
   If $f$ is not exceptional, then we say that $f$ is \emph{ordinary}. 
   \end{definition} 

Our vocabulary might not completely agree with some terms in the established literature. 

\noindent Equipped with Definition~\ref{ordinary} we can now define the notion of a weakly special curve for a split map. 
   \begin{definition}\label{weakly special}
   For this definition we assume that all objects are defined over some algebraically closed field with characteristic zero. Let $C\subset \bP_1^2$ be an irreducible curve and denote its projection into the $i$-th coordinate of $\bP^2_1$ by $\pi_i:C\to \bP_1$ for $i=1,2$. Let $\phi=(f,g):\bP_1^2\to \bP_1^2$ be a split polarized endomorphism of polarization degree $\ge 2$. 
   We say that $C$ is \emph{weakly $\phi$-special} if  one of the following holds: 
   \begin{enumerate}
   \item 
   $\pi_i$ is not dominant for some $i=1,2$; or
   \item 
    $C$ is $\phi$-preperiodic; or 
   \item 
   both $f$ and $g$ are exceptional maps with associated algebraic groups $G_f,G_g$ and maps $\Psi_f, \Psi_g$ and there exists a coset $H \subset G_f\times G_g$ such that $C\cap \mathrm{Im}(\Psi_f, \Psi_g)=(\Psi_f, \Psi_g)(H)$. 
   \end{enumerate}
We say that $C$ is \emph{$\phi$-special} if it is $\phi$-preperiodic or  both $f$ and $g$ are exceptional maps with associated algebraic groups $G_f,G_g$ and maps $\Psi_f, \Psi_g$ and there exists a torsion coset $H \subset G_f\times G_g$ such that $C\cap \mathrm{Im}(\Psi_f, \Psi_g)=(\Psi_f, \Psi_g)(H)$.
   \end{definition}

Determining the equations that define preperiodic curves is a notoriously difficult problem. That said, a precise description of curves that are invariant under the action of split maps acting by a polynomial on each coordinate has been given in the celebrated work of Medvedev and Scanlon \cite{MS}. 
That at hand, Theorem \ref{umb} gives uniform bounds for the number of preperiodic points on translates of a subgroup of $\mathbb{G}_m^2$. 
More precisely, given $n,m \in \bZ\setminus\{0\}$ and an ordinary polynomial $f(z)\in \overline{\bQ}[z]$, Theorem \ref{umb} implies that there is a constant $M=M(n,m)$ (independent of $t$) such that 
 $$\#\{(x,y)\in  \mathrm{Prep}(f) \times  \mathrm{Prep}(f) ~:~x^ny^m = t\}\le M,$$
for all but finitely many $t\in \Qbar$.

   \subsection{Towards a relative dynamical Bogomolov conjecture}    
       
  Theorem \ref{umb} is in fact a corollary of our next result, which gives instances of dynamical analogues of the relative Bogomolov conjecture \cite[Conjecture 1.2]{DGH:assumerbc}. 
    In particular, it provides a better description of the geometric distribution of the sets in \eqref{setfinite}. 
    	\begin{theorem}\label{relative bogomolov}
    		Let $\bbphi:\bP^2_1\to \bP^2_1$ and $\mathbf{\Psi}: \bP^{\ell}_1\to \bP^{\ell}_1$  be isotrivial-free split polarized endomorphisms over $K$ both of polarization degree $d\ge 2$ for $\ell\in \bZ_{\ge 1}$. 
            Let $\bc\subset \bP_1^2$ be an irreducible curve that is not weakly $\bbphi$-special and let $\mathbf{X}\subset \bP^{\ell}_1$ be an irreducible subvariety that is not $\bbpsi$-preperiodic, both defined over $K$. 
    		Then there exists $\epsilon=\epsilon(\bbphi,\bc,\bbpsi,\mathbf{X})>0$ such that 
    		\begin{align*}
    			\{(P,Q)\in(\mathcal{C}\times_{B_0}\mathcal{X})(\Qbar)~:~\hat{h}_{\Phi}(P)+\hat{h}_{\Psi}(Q)<\epsilon\}
    		\end{align*} 
    		is not Zariski dense in $\mathcal{C}\times_{B_0} \mathcal{X}$. Here $\hat{h}_{\Phi}(\cdot), ~\hat{h}_{\Psi}(\cdot)$ denote the corresponding fiber-wise heights from \eqref{fiberwiseheight}.
    	\end{theorem}
         We point out here that in Theorem \ref{relative bogomolov iso} we establish an analogous statement in the case of curves $\mathbf{X}$ and non-isotrivial-free maps. Theorem \ref{relative bogomolov iso} is then used to establish Theorem \ref{umb}.
    	Theorems \ref{relative bogomolov} and \ref{relative bogomolov iso} fit in the theme of unlikely intersections and are inspired by the Pink-Zilber conjectures. We refer the reader to  \cite{Zannier:book} for an overview of related problems in the classical setting. 
    	More generally, the Pink-Zilber conjectures suggest that the following holds.  
    	\begin{conjecture}\label{rbc}
    		Let $\ell\ge 2$ and $\mathbf{F}:\mathbb{P}^{\ell}_1\to \mathbb{P}^{\ell}_1$ be an isotrivial-free split polarized endomorphism over $K$ with polarization degree $\ge 2$. 
    		Let $\mathbf{V}\subset \mathbb{P}^{\ell}_1$ be an irreducible projective subvariety  defined over $K$ and such that $\dim \mathbf{V} <\ell-1$.  
    		Assume that $\mathbf{V}$ is not contained in an $\mathbf{F}$-preperiodic subvariety other than $\bP_1^{\ell}$. 
    		Then there exists $\epsilon=\epsilon(\mathbf{F},\mathbf{V})>0$ such that 
    		\begin{align*}
    			\{P\in\mathcal{V}(\Qbar)~:~\hat{h}_{F}(P) < \epsilon\}
    		\end{align*} 
    		is not Zariski dense in $\mathcal{V}$, where $\hat{h}_{F}(\cdot)$ is the fiber-wise height as in \eqref{fiberwiseheight}.
    	\end{conjecture}
         Theorem \ref{relative bogomolov} is the first instance of Conjecture \ref{rbc} for $\mathbf{V}$ with positive dimension and in a strictly dynamical setting. 
         The only known case of Conjecture \ref{rbc} (and of \cite[Conjecture 1.2]{DGH:assumerbc}) in the case that $\dim \mathbf{V}\ge 1$ is K\"uhne's result \cite[Theorem 1]{kuhne:relative}, who established the conjecture when all components of $\mathbf{F}$ are Latt\`es maps, thereby extending \cite[Theorem 1.4]{DM1}. 
         For related results in the case that $\epsilon=0$ see also \cite{Masser:Zannier,Masser:Zannier:2,Masser:Zannier:nonsimple,Masser:Zannier:simpleA}.

    	That said, there has been a lot of work concerning cases of the conjecture for points $\mathbf{V}$. 
        Then it follows from Baker-DeMarco's conjecture \cite[Conjecture 1.10]{BD:polyPCF}, \cite[Conjecture 6.1]{D:stableheight}. 
        Its first case is established in the breakthrough article of Baker-DeMarco \cite{BD:preperiodic}. 
        A lot of work has centered around other special cases -- see \cite{DWY:Lattes, Ghioca:Hsia:Tucker, GHT:preprint, BD:polyPCF, DWY:Per1, Favre:Gauthier, Favre:Gauthier2, AO1, AO2, AO3}. 
        Most notably, recently Favre and Gauthier \cite{FGpoly} established the conjecture for points $\mathbf{V}\in \bP^2_1(K)$ and split maps $\mathbf{F}:\bP^2_1\to \bP^2_1$ acting by a polynomial on each coordinate.
   
   \smallskip 
   
	\noindent Our next results provide the necessary ingredients towards our proof of Theorem \ref{relative bogomolov}. 
    
    \subsection{Weakly special curves and measures of maximal entropy}
    For each rational map $f\in\bC(z)$ with $\deg f\ge 2$ there is an invariant measure $\mu_f$ supported on the Julia set of $f$ which was constructed by Lyubich \cite{Ly} and Freire-Lopes-Ma\~{n}\'{e} \cite{FLM}. It is the unique measure with no point masses and such that  $f^{*}\mu_f=(\deg f)\mu_f$. 
    Rigidity questions such as when do two rational maps share the same invariant measure have been studied extensively; see \cite{Ye} and the references therein for a relatively recent account.

    Levin and Przytycki \cite[Theorem A]{LP} have shown that the measure $\mu_f$ determines the rational map $f$ up to certain operations. 
    Our next result can be seen as a generalization of \cite[Theorem A]{LP}.
   It allows us to characterize weakly special curves by looking at measures $\mu_f$. 
   
   \begin{theorem}\label{transcendence}
   	Let $f_i\in \bC(z)$ be rational maps of degree $d\ge 2$ for $i=1,2$.
           Let $C \subset \mathbb{P}_1^2$ be an irreducible complex algebraic curve with dominant projections $\pi_i:C\to\bP_1$ for $i=1,2$. Then $\deg(\pi_2)\pi_1^*\mu_{f_1} = \deg(\pi_1)\pi_2^*\mu_{f_2}$
          if and only if $\text{supp}(\pi_1^*\mu_{f_1}) \cap  \text{supp}(\pi_2^*\mu_{f_2}) \neq \emptyset$ and  $C$ is $(f_1, f_2)$- weakly special.
   \end{theorem}
   It is worth remarking that the pull-back of a measure is defined in the sense of distributions. For example if $B \subset C(\C)$ is a Borel set on which $\pi_i$ is injective, then $\pi_i^*\mu_{f_i}(B) = \mu_{f_i}(\pi_i(B))$.
   The condition $\text{supp}(\pi_1^*\mu_{f_1}) \cap  \text{supp}(\pi_2^*\mu_{f_2}) \neq \emptyset$ seems artificial but is  required if both $f_1, f_2$ are exceptional and $G_{f_1} = G_{f_2} = \mathbb{G}_m$. This is also the only case for which it is required. In fact we establish a more general statement, not assuming that $f_1$ and $f_2$ have the same degree; see Theorem \ref{transcendence2}. We note that Theorem \ref{transcendence2} may be seen as an Ax-Lindemann type result for dynamical systems. It states that if we have a linear relation between the measures on an algebraic curve, then this curve has to be weakly special. If the maps $f_1,f_2$ are exceptional, this follows (almost) directly from Ax-Lindemann for algebraic groups, as we will see in our proof.     
   	Assuming further that there exists a point $\la \in C(\mathbb{C})$ such that  $\pi_1(\la)$ is a repelling $f_1$-periodic point  and $\pi_2(\la) $ is $f_2$-preperiodic, a related result is obtained in \cite[Theorem 5.1]{dmm1} and was used to establish the dynamical Bogomolov conjecture. 
   Showing that such a point exists is non-trivial. In \cite{dmm1} the authors infer it from Mimar's main theorem \cite{Mimar} in view of their assumption that $C$ contains infinitely many $(f_1,f_2)$-preperiodic points. But in order to establish the uniform statement in Theorem \ref{umb} we cannot operate under this helpful assumption.
   
   \subsection{A height inequality and a geometric dynamical Bogomolov theorem}\label{introgeom}
	Finally, in Theorem \ref{bogomolov hypothesis}, we establish a  geometric Bogolomov theorem for split maps over $1$-dimensional function fields of characteristic zero. We point out here that Gauthier and Vigny have formulated a general geometric dynamical Bogomolov conjecture \cite[Conjecture 1]{GV} and have reduced it to an elegant complex analytic problem \cite[Theorem B]{GV}.  
    This leads us to an inequality comparing the fiber-wise Call-Silverman canonical height with a height on the base and valid for most points of a non-preperiodic subvariety as in Theorem \ref{splitineq}.  
	For a split polarized endomorphism $\bbphi$ over $K$, we let 
	$$\mathcal{X}^{\bbphi,\star}=\mathcal{X}\setminus\displaystyle\bigcup_{\mathcal{Z}}\mathcal{Z},$$
	where the union runs over all subvarieties $\mathcal{Z}$ of $\mathcal{X}$ that are irreducible components of $\Phi$-special subvarieties projecting dominantly onto the base $B$. 
	
	\begin{theorem}\label{splitineq}
		Let $\ell\ge 2$ and $\bbphi:\mathbb{P}^{\ell}_1\to \mathbb{P}^{\ell}_1$ be an isotrivial-free split polarized endomorphism over $K$ with polarization degree at least equal to $2$. 
		Let $\mathbf{X}\subset \mathbb{P}^{\ell}_1$ be an irreducible projective subvariety  defined over $K$ and let $\mathcal{N}$ be an ample line bundle on $B$. Then exactly one of the following holds:
		\begin{enumerate}
			\item 
			Either $\mathcal{X}^{\bbphi,\star}=\emptyset$ and $\mathbf{X}$ is $\bbphi$-preperiodic
            ; or 
			\item 
			the set 
			$\mathcal{X}^{\bbphi,\star}$ is Zariski open and dense in $\mathcal{X}$ and there are constants $c_1,c_2>0$ depending on $\bbphi$, $\mathbf{X}$ and $\mathcal{N}$ and such that 
			\begin{align}\label{lower bound}
				\hat{h}_{\Phi}(P) \ge c_{1}h_{\mathcal{N}}(\pi(P))-c_{2}, \text{ for all }P\in\mathcal{X}^{\bbphi,\star}(\Qbar).
			\end{align}
		\end{enumerate} 
	\end{theorem}
	
    As an immediate corollary, we obtain uniformity in the dynamical Bogomolov result \cite{dmm2} across $1$-parameter families of isotrivial-free split rational maps and varieties; albeit under the assumption that their height is large enough.

    	\begin{corollary}\label{emptyMM}
    		Let $\ell\ge 2$ and $\bbphi:\mathbb{P}^{\ell}_1\to \mathbb{P}^{\ell}_1$ be an isotrivial-free split polarized endomorphism with polarization degree at least $2$. Let $\mathcal{N}$ be an ample line bundle on $B$.
    		Let $\mathbf{X}\subset \mathbb{P}^{\ell}_1$ be an irreducible subvariety  that is not $\bbphi$-preperiodic. Then there are constants $c_{3}=c_3(\bbphi,\mathbf{X},\mathcal{N})>0$ and $c_4=c_{4}(\bbphi,\mathbf{X},\mathcal{N})>0$ such that 
    		\begin{align}\label{empty}
            \left\{z\in \mathcal{X}_t(\Qbar)\cap\mathcal{X}^{\bbphi,\star}(\Qbar) ~:~\hat{h}_{\Phi_t}(z)\le c_4\max\{h_{\mathcal{N}}(t),1\} \right\}=\emptyset,
            \end{align}
    		for all $t\in B_0(\Qbar)$ with $h_{\mathcal{N}}(t)>c_3$.
    	\end{corollary}

Theorem \ref{splitineq} is a dynamical generalization of Habegger's result \cite[Theorem 1.3]{Habegger:special} which dealt with the case that all components $\bbf_i$ of $\bbphi$ are Latt\`es maps associated to an elliptic curve. 
	It should be noted that a first conditional result in the direction of Theorem \ref{splitineq}, dealing with the case that  $\mathbf{X}$ is a point, is established by Call-Silverman \cite{Call:Silverman}. 
	Under the assumption that the geometric (over the function field $K$) canonical height of $\mathbf{X}$ under $\bbphi$ is non-zero, they derive an inequality such as \eqref{lower bound}. 
	In order to then get the result in Theorem \ref{splitineq} in this case, it is challenging to characterize the points of non-zero geometric canonical height as non-preperiodic. This issue has been clarified by Baker \cite{Baker:finiteness} following work of Benedetto \cite{Ben:finiteness}; see also \cite{D:stableheight}. Therefore, in the case that $\mathbf{X}$ is a point, Theorem \ref{splitineq} follows by combining the results in \cite{Call:Silverman} and \cite{Baker:finiteness}. 
	And to prove Theorem \ref{splitineq} in general, we need to generalize the results in both \cite{Call:Silverman} and \cite{Baker:finiteness} to higher dimensions. 
	To this end we provide a proof generalizing Call-Silverman's result in the appendix; see Theorem \ref{assume bog}. A related result sufficient for our purposes has also been established recently by Yuan and Zhang \cite{Yuan:Zhang:new}. 
	To generalize Baker's result to higher dimensional subvarieties, we establish a geometric dynamical Bogomolov theorem for split maps; see Theorem \ref{bogomolov hypothesis}. 
   
     It is interesting to note that Habegger's proof \cite[Theorem 1.3]{Habegger:special} of Theorem \ref{splitineq} for exceptional maps coming from an elliptic curve  did not rely on the analogous geometric Bogomolov theorem, though the latter had at the time already been established by Gubler \cite{Gubler:split}. 
   Instead, Habegger used local monodromy, the Betti map and Raynaud's theorem on the Manin-Mumford over $\bC$ \cite{Raynaud:1}. As a corollary, he then deduced a new proof of the geometric Bogomolov conjecture for powers of an elliptic curve over $1$-dimensional function fields. 
   This approach was later generalized to families of abelian varieties by Gao and Habegger  \cite[Theorem 1.1]{Gao:Habegger:bog} who thus established the geometric Bogomolov conjecture over $1$-dimensional function fields of characteristic zero as a consequence of a height inequality  \cite[Theorem 1.4]{Gao:Habegger:bog}. The geometric Bogomolov conjecture for function fields of arbitrary dimension was eventually established in \cite{CGH} with a different method.
   Finally  the inequality  \cite[Theorem 1.4]{Gao:Habegger:bog} has been generalized to families of abelian varieties over a base of arbitrary dimension \cite{DGH:uml}. This result coupled with the results of Gao on the mixed Ax-Schanuel conjecture \cite{Gao:generic rank, Gao:mixed}, played a key role in establishing uniformity in the Mordell-Lang conjecture for curves of large height \cite{DGH:uml}. 
   We refer the reader to the survey article \cite{Gao:survey} for a beautiful outline of these developments.

\subsection{A discussion of the proof of Theorems \ref{umb} and \ref{relative bogomolov}}
Let us now briefly describe the strategy of our proof of Theorems \ref{umb} and \ref{relative bogomolov}. 
First we point out that an application of Bézout's theorem allows us to deduce Theorem \ref{umb} from a special case of Theorem \ref{relative bogomolov} and its complimentary Theorem \ref{relative bogomolov iso} in the case that $\bbphi$ has an isotrivial component. We point out  that Dimitrov, Gao and Habegger \cite{DGH:assumerbc} combined the classical relative Bogomolov conjecture \cite[Conjecture 1.2]{DGH:assumerbc} with their work in \cite{DGH:uml} to establish uniformity in the classical Mordell-Lang and Bogomolov problems. 

To prove Theorem \ref{relative bogomolov}, we first use Theorem \ref{splitineq}. 
This allows us to apply the recent arithmetic equidistribution theorem of Yuan and Zhang \cite{Yuan:Zhang:new}; cf. Gauthier's and K\"uhne's equidistribution results \cite{Gauthier:equi, kuhne:uml}. And we do so multiple times -- a very fruitful idea in Bogomolov-type problems that was first employed by Ullmo \cite{Ullmo:Bogomolov}. 
To leverage the output of this equidistribution result, we use the theory of currents as well as a local Hodge index theorem \cite{YZ:index}. Our presentation using currents is mainly inspired by Gauthier and Vigny's work \cite{GV}. Perhaps surprisingly, these tools allow us to reduce Theorem \ref{relative bogomolov} to the complex analytic Theorem \ref{transcendence}, which is an Ax--Lindemann--type result for the fibers of the family. 
This comes in stark contrast with K\"uhne's proofs in \cite{kuhne:relative,kuhne:uml}, which relied on a far more sophisticated `functional transcendence' result, known as the  Ax-Schanuel theorem for mixed Shimura varieties \cite{Gao:mixed}.
This result is not available in the dynamical setting (the proof of Ax-Schanuel uses tame geometry). Consequently, we also obtain a simplified proof of special cases of K\"uhne's aforementioned results. It might be interesting to note that the proof of \cite{DKY:uni}, which built on the strategy from \cite{DKY:ec}, requires careful estimates of a certain energy pairing at each Berkovich space corresponding to a place of $\bQ$. Such estimates are non-trivial, especially in the setting of non-polynomial maps. Our proof bypasses these issues. Indeed, it only requires the complex analytic Theorem \ref{transcendence}, which ensures that certain archimedean energy pairings are non-zero; see Proposition \ref{test function}.
\subsection{Structure of the paper}
The structure of this paper is as follows. 
In \S\ref{background} we briefly recall some background and terminology that we use throughout this article. 
In \S\ref{functional transcendence} we establish Theorem \ref{transcendence} in the stronger form of Theorem \ref{transcendence2}. 
In \S\ref{functional bogomolov} we establish Theorem \ref{splitineq} as well as an analogous result for curves $\mathbf{X}$ and maps $\bbphi$ having an isotrivial component; see Theorem \ref{splitineqiso}. We do so by establishing the geometric dynamical Bogomolov theorems \ref{bogomolov hypothesis} and \ref{mainiso}. 
In \S\ref{fighting the currents} we combine our results to establish Theorem \ref{relative bogomolov} and its complimentary Theorem \ref{relative bogomolov iso}. 
We deduce Theorem \ref{umb} as a consequence of Theorem \ref{relative bogomolov iso} and of a special case of Theorem \ref{relative bogomolov}. 
We further deduce Corollary \ref{small curves}. 
In \S\ref{applications} we establish Theorem \ref{moduli:curve}.
Finally, in the appendix we generalize Call-Silverman's specialization result \cite[Theorem 4.1]{Call:Silverman} to higher dimensional subvarieties.

\smallskip

	\textbf{Acknowledgements:} 
    The first-named author would like to thank Laura DeMarco and Lars K\"uhne for many interesting conversations about topics related to this article over the last few years. We heartily thank Robert Wilms for many enlightening discussions surrounding adelic line bundles. We thank Laura DeMarco, Dragos Ghioca, Thomas Gauthier,  Curtis McMullen, Khoa Nguyen and Gabriel Vigny  for valuable comments and suggestions on a previous version of this paper that greatly improved the current presentation. 
    We would also like to thank the University of Basel for the great working environment it provided during the initial stages of preparation of this article. Myrto Mavraki's research was supported in part with funding from the National Science Foundation. 
Finally, we are grateful to the referees for their very careful reading of the article and their very helpful suggestions. 
	\section{Preliminaries} \label{background}
	
	A notion of adelic metrized line bundles was introduced by Zhang \cite{Zhang small} in the number field case and by Gubler \cite{Gubler:equi} in the function field case. 
	We begin by briefly summarizing the relevant terms for our article. 
	For a more exhaustive introduction we direct the readers to the survey articles of Chambert-Loir \cite{ChambertLoir:survey} and Yuan \cite{Yuan:survey} as well as to the recent book by Yuan and Zhang \cite{Yuan:Zhang:new} and the references therein. 
    An excellent source are also the appendices of \cite{YZ:index}.
	
    First let $\mathfrak{K}$ be an algebraically closed field, complete with respect to a non-trivial absolute value $|\cdot|$.  
    Let $Y$ and $L$ be a projective variety, respectively a line bundle on $Y$ both defined over $\mathfrak{K}$. 
    Let $Y^{\mathrm{an}}$ denote the Berkovich analytification of $Y$ and $L^{\mathrm{an}}$  the analytification of $L$ as a line bundle on $Y^{\mathrm{an}}$. 
    For $x\in Y^{\mathrm{an}}$, we let $\mathfrak{K}(x)$ be its residue field. 
    A \emph{continuous metric} $\|\cdot\|$ on $Y$ is defined to be a collection of $\mathfrak{K}(x)-$norms $\|\cdot\|_x$ one on each fiber $L^{\mathrm{an}}(x)$ for $x\in Y^{\mathrm{an}}$ such that for each rational section $s$ of $L$ the map $x\mapsto  \|s(x)\|_x:=\|s(x)\|$ is continuous on $Y^{\mathrm{an}}$ away from the poles of $s$. The pair $(L,\|\cdot\|)$ is called a (continuous) metrized line bundle on $Y$. A metric is called \emph{semipositive} if it is a uniform limit of semipositive model metrics as in \cite[Definition 5.1]{YZ:index}. Limits of metrics are taken with respect to the topology induced by the supremum norm. 
    A metrized line bundle $L$ on $Y$ is \emph{integrable} if it can be written as $\overline{L}=\overline{L}_1-\overline{L}_2$, where $\overline{L}_i$ are semipositive metrized line bundles on $Y$.
    
    Associated to a semipositive metrized line bundle $\overline{L}$, there is  a curvature form  $c_1(\overline{L})$, defined on  $Y^{\mathrm{an}}$.  One can then define a positive measure $c_{1}(\overline{L})^{\dim(Y)}$ on $Y^{\mathrm{an}}$. In the archimedean places the measure $c_{1}(\overline{L})^{\dim(Y)}$  is constructed by Bedford and Taylor \cite{BedfordTaylor}.
    In the non-archimedean places the curvature forms and associated measures are constructed by Chambert-Loir and Ducros \cite{CLD} and by Chambert-Loir \cite{ChambertLoir:equidistribution} respectively. 
    	We refer the reader to \cite{ChambertLoir:survey} for a more detailed exposition. 
   
    By Chambert-Loir \cite{ChambertLoir:equidistribution}, Gubler \cite{Gubler:local} and Chambert-Loir-Thuillier \cite{CLT} we can define local heights of subvarieties $Z$ of $Y$, though not uniquely, which can be computed through integration against measures obtained from the aforementioned curvature forms. More precisely, if  $n=\dim Z$, and for $i=0,\ldots,n$ we let $\overline{L}_i$ be integrable metrized line bundles with sections $\ell_i$ having disjoint support, then the local height of $Z$ with respect to $(\ell_0,\ldots,\ell_n)$ can be defined inductively through 
    	\begin{align}
    		\begin{split}
    			\widehat{\di(\ell_0)}\cdots\widehat{\di(\ell_n)}\cdot [Z]&:=\widehat{\di(\ell_1)}\cdots\widehat{\di(\ell_n)}\cdot [\di(\ell_0)\cdot Z]\\&-\int_{Y^{\mathrm{an}}}\log\|\ell_0\|c_1(\overline{L}_1)\cdots c_1(\overline{L}_n)\delta_{Z}.
    		\end{split}
    	\end{align}
    	A summary for intersection numbers sufficient for our purposes can be found in \cite{YZ:index}. 
    	We point out here that in the special case that $Y=C$ is a curve and $\overline{L}_0,\overline{L}_1$ are defined on trivial line bundles $L_0=L_1=\mathcal{O}_C$, we may use the regular sections $1_{L_0}$ and $1_{L_1}$ to define the local  intersection number of the metrics of $L_0$ and $L_1$, as follows
    	\begin{align}\label{local pairing}
    		\overline{L}_0\cdot \overline{L}_1:=\widehat{\di(1_{L_0})}\cdot\widehat{\di(1_{L_1})}\cdot [C].
    	\end{align}
    
    Now let $\mathcal{F}$ be either a number field or the function field of a smooth projective curve defined over $\Qbar$. 
	Denote by $M_{\mathcal{F}}$ the set of places $v$ of $\mathcal{F}$ normalized so that the product formula 
	\begin{align}
		\displaystyle\prod_{v\in M_{\mathcal{F}}} |x|^{n_v}_v=1
	\end{align}
	holds for each nonzero $x\in \mathcal{F}$. Here $n_v=[\mathcal{F}_v:\mathbb{Q}_v]/[\mathcal{F}:\mathbb{Q}]$ if $\mathcal{F}$ is a number field, whereas $n_v=1$ in the function field case. For each $v\in M_{\mathcal{F}}$ we denote by $\mathcal{F}_v$ a completion of $\mathcal{F}$ with respect to $v$ and write  $\overline{\mathcal{F}_v}$ for its algebraic closure and $\bC_v$ for its completion.
	
	Let $Y$ be an irreducible projective variety over $\mathcal{F}$ and $L$ a line bundle on $Y$. 
	A (continuous) \emph{semipositive adelic metric} on $L$ is a collection $\{\|\cdot\|_v~:~v\in M_{\mathcal{F}}\}$ of  continuous semipositive $\mathbb{C}_v$-metrics $\|\cdot\|_v$ on $L$ over all places $v$ of $\mathcal{F}$ satisfying certain coherence conditions; \cite{Yuan:survey, Gubler:equi}. 
	We denote the pair of our line bundle and the family of metrics by 
	$$\overline{L}:=(L, \{\|\cdot\|_v\})$$ and call $\overline{L}$ an \emph{adelic (semipositive) metrized line bundle}. 
    An adelic metric on $L$ is called \emph{integrable} if it is the difference of two semipositive adelic metrics. 
    	Following \cite{YZ:index}, we denote the space of integrable adelic metrized line bundles on $Y$ by $\widehat{\mathrm{Pic}}(Y)_{\mathrm{int}}$. 

	For this article, the most important example of an adelic semipositive metrized line bundle is the so-called \emph{$\phi$-invariant metric} $\overline{L}_{\phi}$ associated to an endomorphism $\phi:Y\to Y$ defined over $\mathcal{F}$ with $\phi^{*}L\simeq L^{\otimes d}$ for some $d\ge 2$. This metric was defined by Zhang \cite{Zhang small} following Tate's limiting argument. 
   
	Zhang \cite[Theorem (1.4)]{Zhang small}, Gubler \cite{Gubler:local,Gubler:split, Gubler:equi} and Chambert-Loir-Thuillier \cite{CLT} developed a height theory for subvarieties and an intersection theory for integrable adelic metrized line bundles for number fields and function fields respectively. 
	So denoting the set of $n\ge 0$ dimensional cycles of $Y$ by $Z_n(Y)$, their respective intersection numbers are written as
	\begin{align}\label{intersection}
		\begin{split}
			\widehat{\mathrm{Pic}}(Y)_{\mathrm{int}}^{n+1}\times Z_n(Y)&\to \mathbb{R}\\
			(\overline{L}_0,\cdots,\overline{L}_n,D)&\mapsto \overline{L}_0\cdots\overline{L}_n\cdot D=\overline{L}_0|_D\cdots\overline{L}_n|_D.
		\end{split}
	\end{align}
	When $D=Y$ we omit it from the above notation. 
	
	Using the intersection numbers in \eqref{intersection} one can give a notion for the height of a subvariety associated to any $\overline{L}\in \widehat{\mathrm{Pic}}(Y)_{\mathrm{int}}$. More specifically, if $Z$ is an irreducible closed subvariety of $Y$ defined over $\overline{\mathcal{F}}$, we denote by $Z_{\mathrm{gal}}$ the closed $\mathcal{F}$-subvariety of $Y$ corresponding to $Z$ and write 
	
	\begin{align}\label{height def}
		h_{\overline{L}}(Z)=\frac{\overline{L}^{\dim Z+1}\cdot Z_{\mathrm{gal}}}{(\dim Z+1)\deg_{L}(Z_{\mathrm{gal}})}\in \bR.
	\end{align}
	In the case that $Z=\{x\}$ is just a point in $Y(\overline{\mathcal{F}})$, one has 
	\begin{align}
		h_{\overline{L}}(x)=-\frac{1}{|\mathrm{Gal}(\overline{\mathcal{F}}/\mathcal{F})\cdot x|}\sum_{v\in M_{\mathcal{F}}}\sum_{z\in \mathrm{Gal}(\overline{\mathcal{F}}/\mathcal{F})x}n_v\log\|s(z)\|_v,
	\end{align}
	where $s$ is meromorphic section of $L$ that is regular and non-vanishing at $x$. 
 
	As an example, in the case of the $\phi$-invariant metric $\overline{L}_{\phi}$, the height
	$h_{\overline{L}_{\phi}}$ is the canonical height $\hat{h}_{L,\phi}$ of $\phi$  introduced by Call-Silverman \cite{Call:Silverman} taken with respect to the line bundle $L$.
   
	For $i=1,\ldots, \dim Y+1$ we have the $i$th successive minimum 
	\begin{align}\label{minima}
		e_{i}(Y,\overline{L})=\sup_Z\inf_{P\in Y(\overline{\mathcal{F}})\setminus Z} h_{\overline{L}}(P),
	\end{align}
	where $Z$ ranges over all closed subsets of codimension $i$ in $Y$.
	Zhang \cite{Zhang small} and Gubler \cite[Lemma 4.1, Proposition 4.3]{Gubler:split} in the setting of number fields and function fields respectively proved a fundamental inequality establishing a relationship between these sucessive minima and the height of $Y$. 
	\begin{lemma}[Zhang-Gubler fundamental inequality]\label{zhang ineq}
     Assume that $\overline{L}$ is a semipositive continuous adelic metrized line bundle on an irreducible projective $d$-dimensional variety $Y$ defined over $\mathcal{F}$. Assume that $L$ is ample. Then 
		$$\frac{e_{1}(Y,\overline{L})+\cdots+e_{d+1}(Y,\overline{L})}{\dim Y+1}\le h_{\overline{L}}(Y)\le e_{1}(Y,\overline{L}).$$
	\end{lemma}
	
    We finish this section with a convention. 
    Throughout this article when $\phi=(f_1,\ldots,f_{\ell}):\mathbb{P}^{\ell}_1\to \mathbb{P}^{\ell}_1$ is a split polarized endomorphism over $\mathcal{F}$ with polarization degree $d\ge 2$ we will be working with the line bundle $L_0$ associated to the divisor 
        $$\{\infty\}\times\bP_1\cdots\times\bP_1+\cdots+\bP_1\times\cdots\times\mathbb{P}_1\times\{\infty\},$$ giving a polarization. We will abbreviate the height from \eqref{height def} in this case by
        \begin{align}\label{height convention}
        \hat{h}_{\phi}(\cdot):=h_{(\overline{L_0})_{\phi}}(\cdot).
        \end{align} 
    
	
    	\section{Equal measures and weakly special curves}\label{functional transcendence}
    Our first goal in this section is to prove Theorem \ref{transcendence}. In fact we will prove a more general result applicable to non-polarized endomorphisms. Then in \S\ref{local index} we apply Theorem \ref{transcendence} to construct an auxiliary continuous real-valued function on $B\times\bP^2_1$ which can be used to detect weakly special fibers in families of curves. 
   
         Let 
          $f\in\bC(z)$ be a rational map with $\deg f\ge 2$. Recall that Lyubich \cite{Ly} and Freire-Lopes-Mañé \cite{FLM} have constructed an invariant measure $\mu_f$ supported on the Julia set of $f$, $J_f$. 
          
               We will show that the following holds.
        \begin{theorem}\label{transcendence2}
        	Let $f_i\in \bC(z)$ be rational maps of (not necessarily equal ) degree $\deg(f_i)\ge 2$ for $i=1,2$. 
            Let $C \subset \mathbb{P}_1^2$ be an irreducible complex algebraic curve with dominant projections $\pi_i:C\to\bP_1$ for $i=1,2$. 
            Then $\deg(\pi_2)\pi_1^*\mu_{f_1} = \deg(\pi_1)\pi_2^*\mu_{f_2}$ 
                   if and only if $\text{supp}(\pi_1^*\mu_{f_1}) \cap \text{supp}(\pi_2^*\mu_{f_2}) \neq \emptyset $ and there exist positive integers $\ell, k$ such that $\deg(f_1)^{\ell} = \deg(f_2)^{k}$ and $C$ is weakly $(f_1^{ \ell }, f_2^{k})$-special. 
        \end{theorem}
    
    Our proof relies on results from complex dynamics. For basic notions and results we refer the reader to \cite{milnor}. We will use Koebe's 1/4-theorem, which we recall here for the reader's convenience.  For $a\in \bC$ and $r\in \bR$ we write 
    $$B(a,r)=\{z\in \bC~:~|z-a|<r\}.$$ 
    \begin{theorem}[Koebe 1/4-theorem]
    Suppose that $f:B(a,r)\to \bC$ is a univalent function. Then 
    $f(B(a,r))\supseteq B(f(a), \frac{r}{4}|f'(a)|)$.
    \end{theorem}

    Let $J_f$ be the Julia set of a rational function $f\in\bC(w)$ with $\deg f\ge 2$.  
    A holomorphic function $h:U \rightarrow \mathbb{C}$ defined on a neighborhood $U$ of $z_0\in J_f$ is called a \emph{symmetry} of $J_f$ if 
    \begin{itemize}
    \item 
   $z\in U\cap J_f\Leftrightarrow h(z)\in h(U)\cap J_f$; and 
    \item 
    if $J_f$ is a circle, the entire Riemann sphere or a line segment, there is a constant $\al>0$ such that for any Borel set $A$ on which $h|_{A}$ is injective we have $\mu_f(h(A))=\al\mu_f(A)$.
    \end{itemize}
    A family $\mathfrak{S}$ of symmetries of $J_f$ is called \emph{non-trivial} if $\mathfrak{S}$ is normal  \textbf{and} no limit function of $\mathfrak{S}$ is constant. The following result is classical and will be central in our proof. 
    It combines Levin's results \cite{Levin} as well as the work of Douady and Hubbard  \cite{DH}, which relied on ideas by Thurston.
    
      		\begin{theorem}[\cite{DH},\cite{Levin}]\label{Levin}
              If the Julia set of a rational map $f$ has an infinite non-trivial family of symmetries, then $f$ is exceptional.
      		\end{theorem}
      	 
    Next we record here the following lemma, which follows from  \cite[Lemma 1]{LP}.  Note that \cite[Lemma 1]{LP} is essential in Levin-Przytycki's proof of Theorem \ref{transcendence2} for the case of the diagonal curve  \cite{LP}.

\begin{lemma}\label{LP:key}
Let $f$ be an ordinary rational function of degree $\ge 2$ defined over $\bC$. 
    Then there exists a set $E_f\subset J_f$ with $\mu_f(E_f)>\frac{1}{2}$ and numbers $r_{f}>0, T_f>0$ and $\kappa_f>1$, 
    so that the following holds. For each $x\in E_f$ there exists a set $R_f(x)\subset \mathbb{R}$, such that for every $T>T_f$, the set $R_{f,T}:=R_f(x)\cap [0,T]$ has Lebesgue measure at least $5/8T$ and for every $\tau\in R_f(x)$ there exists $n\in \mathbb{N}$ such that the following hold: 
\begin{enumerate}
\item     $f^{n}: B(x,e^{-\tau})\to \bP_1$ is injective; and we have 
\item $ B(f^{n}(x), r_f/\kappa_f)\subset f^{n}(B(x, e^{-\tau}))\subset B(f^{n}(x), r_f).$
\item $n\to\infty$ as $\tau\to\infty$. 
\end{enumerate}
\end{lemma}

\begin{proof}
  The measure $\mu_{f}$ is ergodic for $f$ and its Lyapunov exponent $\chi\left(\mu_{f}\right)$ is positive - it even satisfies $\chi\left(\mu_{f}\right) \geq \frac{1}{2} \log d$ (see \cite{Ly,FLM,Ru} or \cite{BD}). Therefore, we can apply \cite[Lemma 1]{LP} and the lemma follows.
    \end{proof}

    We are now ready to prove Theorem \ref{transcendence2}. 
           
     \noindent\emph{Proof of Theorem \ref{transcendence2}:}
     Let $f_i$, $C$ and $\pi_i$ be as in the statement of the theorem. We will first prove the direct implication. 
     We have
          \begin{align}\label{equal measures}
              \mu_{C} := \frac{\pi_1^*\mu_{f_1}}{\deg(\pi_1)} = \frac{\pi_2^*\mu_{f_2}}{\deg(\pi_2)}
              \end{align}
              and in particular
              \begin{align}\label{Julia}
              \pi_1^{-1}(J_{f_1})=\pi_2^{-1}(J_{f_2}).
              \end{align}
              The equation \eqref{Julia} follows from the fact that the Julia set of $f$ is the support of $\mu_f$. A consequence of (\ref{equal measures}) and the property of the pull-back measure mentioned just after Theorem \ref{transcendence} is  that for any holomorphic branch $s:U \rightarrow \mathbb{P}_1(\mathbb{C})$ of $\pi_2\circ\pi_1^{-1}$  on an open Borel set $U$, we have the following equality 
              \begin{align}\label{analyticsection}
                  \mu_{f_2}(s(B)) = \frac{\deg(\pi_2)}{\deg(\pi_1)}\mu_{f_1}(B).
              \end{align}
              
              As explained in the proof of \cite[Theorem 2.2]{Mimar}, equation \eqref{Julia} implies that $J_{f_1}$ is a smooth set (a circle, a line segment or the entire sphere) if and only if $J_{f_2}$ is.
          Moreover by \cite[Section 9]{DH} and \cite{Zd} a rational function $f$ with degree $\ge 2$ is exceptional if and only if its Julia set is smooth and $f$ has a smooth measure of maximal entropy. 
          Therefore, by \eqref{equal measures} and \eqref{Julia} we infer that $f_1$ is ordinary if and only if $f_2$ is. 
                    
        Assume first that both $f_i$ are ordinary and let $E_i:=E_{f_i}\subset J_{f_i}$, $r_i:=r_{f_i}>0$ and $\kappa_i:=\kappa_{f_i}>1$ be the quantities associated to $f_i$ defined in Lemma \ref{LP:key} for $i=1,2$. 
        Since $\mu_C(\pi_i^{- 1}(E_i)) = \mu_{f_i}(E_i)>\frac{1}{2}$ for $i=1,2$ we have 
        		\begin{align}\label{positive measure}
                \mu_C(\pi_1^{-1}(E_1)\cap \pi_2^{-1}(E_2)) > 0.
                \end{align}
                	Let $\mathcal{R}\subset C(\mathbb{C})$ be the (finite) set of ramification points of $\pi_1$ and $\pi_2$. 
                    As $\mu_C$ does not charge points equation \eqref{positive measure} ensures that there is $p_0 \in \pi_1^{-1}(E_1)\cap \pi_2^{-1}(E_2)\setminus \mathcal{R}$.

Applying Lemma \ref{LP:key} to $\pi_i(p_0)\in E_i\setminus \pi_i(\mathcal{R})$, since the set $R_{f_1}(\pi_1(p_0))\cap (R_{f_2}(\pi_2(p_0))+1)\cap [0,T]$ has Lebesgue measure at least $1/9T$ for $T$ large enough, we can find strictly increasing sequences 
                    $t^i_n\to \infty, t_n^2 = t_n^1-1$ and $k^i_n\to \infty$ as $n\to\infty$ so that for each $n\in\bN$,  
 the map 
                    \begin{align}\label{inj}
                    f_i^{k^i_n}&: B(\pi_i(p_0), e^{-t^i_n})\to \bP_1
                    \end{align}
                    is injective, $B(\pi_i(p_0), e^{-t_n^i})\cap \pi_i(\mathcal{R})=\emptyset$ and 
                    \begin{align}\label{lp1}
                    B(f^{k^i_n}_i(\pi_i(p_0))  , r_i/\kappa)\subset f^{k^i_n}_i(B(\pi_i(p_0), e^{-t^i_n}))\subset B(f^{k^i_n}_i(\pi_i(p_0)), r_i),
                    \end{align}
                    where $\kappa:=\max\{\kappa_1,\kappa_2\}>1$.

We set $s$ to be a branch of $\pi_2\circ \pi_1^{-1}$ near $\pi_1(p_0)$ and we may assume without loss of generality that 
                     \begin{align}\label{wlog}
                     |s'(\pi_1(p_0))| \leq 1.
                     \end{align}
                       Indeed we either have $|(\pi_2\circ \pi_1^{- 1})'(\pi_1(p_0))| \leq  1$ or $|(\pi_1\circ \pi_2^{- 1})'(\pi_2(p_0))| \leq  1$ and if the latter happens we may just reverse the roles of $f_1,f_2$ in what follows. 
                      By \eqref{wlog} we may also assume that
                        \begin{align}\label{wlog2}
                        \sup_{n\in\bN}\sup\{|s'(z)| ~:~z \in B(\pi_1(p_0), e^{-t^1_n})\}\le 2,
                        \end{align}
                    
                    As the Julia set $J_{f_i}$ is compact (in the topology induced by the chordal distance on $\bP_1(\bC)$) passing to a subsequence we may also assume that 
                    \begin{align}\label{bounded}
                   f^{k^i_n}_i(\pi_i(p_0))\to p_i,
                    \end{align}
                     for some $p_i\in J_{f_i}$ as $n\to\infty$. 
                    Passing to a further subsequence  we can thus infer that
                                          \begin{align}\label{balls} 
                                               	B(f_1^{k^1_n}(\pi_1(p_0)), c_2r_1) \subset B(p_1,c_1r_1) \subset  f_1^{k^1_n}(B(\pi_1(p_0), e^{-t^1_n})),
                                               	\end{align} 
                                               for positive constants $c_1,c_2>0$  and all $n\in\bN$. 
                      From \eqref{balls}, \eqref{wlog2} and $t_n^2 = t_n^1 -1$ we infer that 
                                                    \begin{align}
                                                    s\circ f_1^{-k^1_n} (B(p_1,c_1r_1)) \subset  B(\pi_2(p_0), e^{-t^2_n}),
                                                    \end{align}
                                                    Here we have made a choice of inverse branch of $f_1^{k^1_n}$.
                                                    Thus by \eqref{lp1} and \eqref{bounded} after perhaps passing to another subsequence, it also holds that
                                                   \begin{align}\label{is contained in ball}
                                                    f_2^{k^2_n} \circ s\circ f_1^{-k^1_n} B(p_1,c_1r_1)\subset B(p_2, 2r_2).
                                                   \end{align}
         By Montel's theorem we thus infer that the functions
              $$\tilde{h}_n := f_2^{k^2_n} \circ s\circ f_1^{-k^1_n}: B(p_1,c_1r_1)\to \bP_1(\bC)$$
              form a normal family. 
              
              Next, we will show that $\{\tilde{h}_n\}_n$ has no subsequence that converges to a constant function.                                           
        Applying the Cauchy integral formula to $f_1^{k^1_n}$ at $\pi_1(p_0)$ and using \eqref{lp1} and \eqref{bounded} we deduce 
            \begin{align}\label{lowerd}
            |(f_1^{-k^1_n})'(f_1^{k^1_n}(\pi_1(p_0)))|\geq c_3r^{-1}_1e^{-t^1_n},
            \end{align} 
             for an constant $c_3> 0$ and all $n\in\bN$.   
           From (\ref{balls}), inequality (\ref{lowerd}) and Koebe's 1/4-theorem we deduce
         \begin{align}\label{koebe1}
         B(\pi_1(p_0), c_4e^{-t^1_n}) \subset f^{-k^1_n}(B(p_1, c_1r_1))
         \end{align}
         for some $c_4>0$.
         By Koebe's 1/4-theorem and the fact that $s'(\pi_1(p_0)) \neq 0$ we also have that 
         \begin{align}\label{koebe2}
         B(\pi_2(p_0), c_5e^{-t^2_n})\subset s(B(\pi_1(p_0), c_4e^{-t^1_n}))
         \end{align}
         for some $c_5>0$ and all $n\in\bN$. 
         Similarly, applying the Cauchy integral formula to $f^{-k^2_n}_2$ at $f^{k^2_n}_2(\pi_2(p_0))$ and using \eqref{lp1} and \eqref{bounded} we get 
         \begin{align}\label{lowerboundkoebe}
           |(f_2^{k^2_n})'(\pi_2(p_0))|\geq c_6r_2e^{t^2_n},
           \end{align}
           for $c_6>0$ and all $n\in\bN$. 
         In view of \eqref{lowerboundkoebe}, Koebe's 1/4-theorem combined with \eqref{koebe1} and \eqref{koebe2} implies that
         \begin{align}\label{contains ball}
          B(f_2^{k^2_n}(\pi_2(p_0)), c_7r_2) \subset f_2^{k^2_n} \circ s\circ f_1^{-k^1_n}(B(p_1,c_1r_1))
          \end{align}
         for a positive constant $c_7$ (we used that $t_n^2 =t_n^1 -1$). This proves our claim that $\{\tilde{h}_n\}_{n\in\bN}$ has no subsequence which converges to a constant function.
         
         Our next task is to construct a non-trivial family of symmetries of $J_{f_1}$. 
         To this end let $h^*$ be a limit function of the family $\{\tilde{h}_n\}_{n\in\bN}$. 
         We can choose a ball $B^{*}:=B(p^*, r^*) \subset B(p_1,c_1r_1)$ centered at $p^{*}\in J_{f_1}$ such that $h^*(B(p^*, 2r^*)) \cap \pi_2(\mathcal{R}) = \emptyset$ (indeed recalling that $\pi_1(p_0)\in J_{f_1}$ and that \eqref{bounded} holds we can choose $p^{*}= f^{k^1_n}_1(\pi_1(p_0))$ for some large $n$.). 
         As restrictions of sequences of functions converge to the restriction of  the limit function, passing to a subsequence of $\tilde{h}_n$ we may assume that $\tilde{h}_n(B^{*}) \cap \pi_2(\mathcal{R}) = \emptyset$ for all $n\in\bN$ and $\tilde{h}_n \rightarrow h^*$ (which is non-constant). Now consider the functions defined by 
         \begin{align}
         h_n:= s^{-1}\circ \tilde{h}_n: B^{*}\to \bP_1(\bC).
         \end{align}
         As we have seen, they form a normal family with no subsequence that converges to a constant function. Moreover if $A \subset B^{*}$ is a Borel measurable set (on which $h_n$ is injective), then because of our assumption \eqref{equal measures} and in particular (\ref{analyticsection}) we have
         \begin{align}\label{measuresequence} 
         	\mu_{f_1}(h_n(A)) = \frac{\deg(f_2)^{k^2_n}}{\deg(f_1)^{k^1_n}}\mu_{f_1}(A).
         \end{align}
         Thus we have a non-trivial family of symmetries of $J_{f_1}$. 
         As $f_1$ is ordinary,  by Theorem \ref{Levin} we infer that $\{h_n\}$ is eventually a constant sequence. 
         Thus there exist positive integers $k,\ell$ and non-negative integers $N,M$ such that 
          \begin{align}\label{equation}
           f_2^{M + k}\circ s \circ f_1^{-(N + \ell)} = f_2^{M }\circ s \circ f_1^{-N }.
           \end{align}
        From \eqref{measuresequence} it follows directly that $\deg(f_2)^{k}=\deg(f_1)^{\ell}$. 
        We are now going to show that 
        \begin{align}\label{prep curve}
         (f_1^{N + \ell},f_2^{M+ k})(C)= (f_1^{N},f_2^{M})(C).
        \end{align}
         Assuming \eqref{prep curve} holds, as  in \cite[Lemma 6.1]{dmm1} it follows that if $\pi_1(p)$ is preperiodic by $f_1$ for a point $p \in C(\mathbb{C})$ then $\pi_2(p)$ is preperiodic by $f_2$. 
         Thus we can choose $p$ such that $\pi_1(p)$ is a repelling periodic point of $f_1^{ \ell}$ that is neither a ramification point of $\pi_1$ nor of any iterate of $f_1$. Then $\pi_2(p)$ is also preperiodic by $f_2$. Taking $L\in\bN$ such that $f_2^{ Lk}(\pi_2(p_2))$ is fixed by $f_2^{ Lk}$ and $\pi_1(p)$ is fixed by $f_1^{ L\ell}$, we see that $(f_1^{ L\ell},f_2^{ Lk})(C)$ satisfies the conditions of  \cite[Theorem 5.1]{dmm1} and the conclusion of our theorem follows. 
         
         To prove \eqref{prep curve} consider the image of an arc
         $$\mathcal{I} := \{ (f_1^{-(N+ \ell)}(x), s\circ f_1^{-(N+\ell)}(x))~:~x\in B^{*}\}$$
         which is well-defined by our construction. 
         Since the inverse $f^{-N}$ is well-defined it holds that 
         \begin{align}
         \begin{split}
          (f_1^{N + \ell},f_2^{ M + k})(\mathcal{I}) &= \{(x, f_2^{M}\circ s\circ f_1^{-N}(x))~:~x \in B^{*}\}\\
          &= \{(f_1^{N}\circ f_1^{-N}(x), f_2^{M}\circ s\circ f_1^{-N}(x))~:~x \in B^{*}\}\\
          &=\{(f_1^{N}(y), f_2^{M}\circ s(y))~:~ y \in f_1^{-N}(B^{*})\}.
          \end{split}
          \end{align}
       We infer that $(f_1^{N + \ell},f_2^{ M + k})(\mathcal{I}) \subset  (f_1^{N},f_2^{ M})(C)$. 
       It is also evident that $(f_1^{N + \ell},f_2^{M + k})(\mathcal{I}) \subset (f_1^{N + \ell},f_2^{ M + k})(C)$ since $\mathcal{I}\subset\mathcal{C}$. Thus, 
        $$(f_1^{N + \ell},f_2^{ M + k})(\mathcal{I})  \subset  (f_1^{N},f_2^{ M})(C)\cap(f_1^{N + \ell},f_2^{ M + k})(C) $$ 
        and by Bézout's theorem and the fact that $C$ is irreducible our claim \eqref{prep curve} follows. 
        This finishes the proof in the case that both $f_1,f_2$ are ordinary. \\
        
        If now both $f_1$ and $f_2$ are exceptional then the statement becomes more classical and should be known. We have not found a proof in the literature and thus provide a quick sketch here. 
        First note that if one of the groups, say $G_{f_1}$  is an elliptic curve then the other has to be as well as the support of $\mu_{f_1}$ is $\mathbb{P}_1(\mathbb{C})$ and if $G_{f_2} = \mathbb{G}_m$ then the support is a topological circle or an interval. 
        If both groups $G_{f_1}, G_{f_2}$ are elliptic curves then we can employ the arguments in \cite{DM2}. In place of Bertrand's theorem \cite[Th\'eor\`em 5]{Bertrand} we can use the main theorem in \cite{Ax}. 
        Assume now that  $G_{f_1} = G_{f_2} = \mathbb{G}_m$. Then $\Psi_{f_i}^*\mu_{f_{i}}$ is a rational multiple of the Haar measure on the circle $S^1$, denoted by $\mu_{S^1}$ for $i = 1,2$. 
        Let $H$ be an irreducible component of $(\Psi_{f_1}, \Psi_{f_2})^{-1}(C)$ and let $\tilde{\pi}_1, \tilde{\pi_2}$ be the two canonical projections from $H \subset \mathbb{G}_m^2$ to $\mathbb{G}_m$. The first assumption of the theorem  then implies that $\tilde{\pi}_1^*\mu_{S^1} = \alpha \tilde{\pi}_2^*\mu_{S^1}$  for $\alpha \in \mathbb{Q}^*$. It follows that there exists a non-constant analytic curve $\gamma:(0,1) \rightarrow H$ such that $\mathrm{Im}(\tilde{\pi}_i\circ \gamma) \subset S^1$ for $i =1,2$ and such that $\tilde{\pi}_1, \tilde{\pi}_2$ are injective on the support of $\gamma$ 
        and that 
        $$\int_{1/2}^{t}dg_1 = \alpha\int_{1/2}^{t}dg_2$$
       where $g_i = \log \tilde{\pi}_i \circ \gamma, i =1,2$ for all $t $ in a neighborhood of $1/2$.  We infer that  $\log \tilde{\pi}_i \circ \gamma(t)$ for $i = 1,2 $ are linearly related over $\mathbb{Q}$ modulo a constant function. By Ax's theorem \cite{Ax} we infer that $H$ is a coset; a weaker Ax--Lindemann version would suffice.\\
  
        We will now prove the converse implication. 
        Assume that the curve $C$, which projects dominantly into both coordinates, is weakly $(f_1,f_2)$-special. We will show that \eqref{equal measures} holds. 
         First assume that both maps $f_1$ and $f_2$ are ordinary so that $\deg(f^k_1)=\deg(f^{\ell}_2)$ and $C$ is $(f^k_1,f^{\ell}_2)$-preperiodic for $k,\ell\ge 1$.  
        If both maps $f_i$ and $C$ are defined over $\Qbar$, then the result is known to  hold. 
        Indeed by \cite[Lemma 6.1]{dmm1} the curve $C$ contains infinitely many $(f^k_1,f^{\ell}_2)$-preperiodic points. Then \eqref{equal measures} follows by \cite[Proposition 4.1]{dmm1}. 
        Otherwise, assume that $f_1,f_2$ and $C$ are defined over the function field $K=\Qbar(V)$ of a variety $V$ defined over $\Qbar$ with positive dimension. We can spread the $f_i$ and $C$ out to get endomorphisms 
        $f_{i,U}:U\times\bP_1\to U\times\bP_1$ and a family of curves $\mathcal{C}\subset U\times\bP^2_1$ on a Zariski open and dense subset $U$ of $V$ with the property that fiberwise $f_{i,t}$ has degree $\deg f_i$ and $\mathcal{C}_t$ is irreducible for $t\in U$ and $i=1,2$. 
        Then for each $t\in U(\Qbar)$, the maps $f_{i,t}$ and the curve $\mathcal{C}_t$ are defined over $\Qbar$ and by \cite[Lemma 6.1]{dmm1} and \cite[Proposition 4.1]{dmm1} we infer that 
        \begin{align}\label{equalmeasures2}
        \frac{\pi_{1,t}^*\mu_{f_{1,t}}}{\deg(\pi_{1,t})} = \frac{\pi_{2,t}^*\mu_{f_{2,t}}}{\deg(\pi_{2,t})}.
        \end{align}
        Since $U(\Qbar)$ is dense in $V(\bC)$ in the euclidean topology and the measures above vary continuously, we infer that \eqref{equalmeasures2} holds for all $t\in V(\bC)$. 
        In particular, \eqref{equalmeasures2} holds for the point $t$ giving the embedding of $\Qbar(V)$ into $\bC$ and the conclusion follows.

        Assume now that $f_1$ and $f_2$ are exceptional. 
        The existence of a one dimensional coset dominating both factors implies that $G_{f_1}, G_{f_2}$ are either both elliptic curves or both equal to $\mathbb{G}_m$. 
        If they are both elliptic curves, then it is clear from the definition of the Betti coordinates that they are linearly related and equation \eqref{equal measures} follows once we express the measures in terms of the Betti coordinates as in \cite[\S 4 Addendum]{CDMZ}. 
        If $G_{f_1} = G_{f_2} = \mathbb{G}_m$ then by our assumption we have a coset $H$ in $\mathbb{G}^2_m$ with $C=(\Psi_{f_1},\Psi_{f_2})(H)$. Denoting its two projections by $\tilde{\pi}_i: H\to \mathbb{G}_m$, for $i = 1,2$, it suffices now to show that $\tilde{\pi}_i^*\mu_{S^1}$ are linearly related for $i=1,2$. 
        Our assumption that $\text{supp}(\pi_1^*\mu_{f_1}) \cap \text{supp}( \pi_2^*\mu_{f_2}) \neq \emptyset$ ensures that on $H$ it holds that $\tilde{\pi}_1(z)^{k_1}\tilde{\pi_2}(z)^{k_2} = \theta$ for two non-zero integers $k_1,k_2$ and $\theta \in S^1$. In particular, $k_1\log^{+}|\tilde{\pi}_1|= k_2\log^{+}|\tilde{\pi}_2|$ and \eqref{equal measures} follows. This completes the proof of our theorem.  
         \qed
    
    \subsection{Applications of Theorem \ref{transcendence}}\label{local index}
    
    We proceed to discuss the consequence of Theorem \ref{transcendence} which will be crucial in our proof of Theorem \ref{relative bogomolov} (hence also Theorem \ref{umb}.) 
    
    For a curve $\bc\subset \bP_1\times\bP_1$ we let $\pi_i:\bc \to \bP_1$ for $i=1,2$ denote its projections into the first and second coordinate. For each $i$ and $t\in B_0(\bC)$ we write $\pi_{i,t}: \mathcal{C}_{t}\to \bP_1$ for the corresponding specialization in the fiber $\mathcal{C}_t$. 
    We denote by $|\cdot|$ the standard absolute value in $\bC$ and for $x,y\in\bC$ we write 
    $$\|(x,y)\|=\max\{|x|,|y|\}.$$
    
    \begin{definition}\label{finite set}
    Let $\bbphi=(\bbf_1,\bbf_2):\bP_1^2\to \bP_1^2$ be a split  endomorphism and let $\bc\subset \bP_1^2$ be an irreducible curve, both defined over $K$. We denote by $E_{\bbphi,\bc}\subset B_0(\Qbar)$ the smallest finite set such that for all $t\in B_0(\bC)\setminus E_{\bbphi,\bc}$ the following hold:
    \begin{enumerate}
    \item
    $\mathcal{C}_t$ is an irreducible  curve of genus equal to the genus of $\bc$; 
    \item 
    $\deg(\pi_{i,t})=\deg(\pi_{i})$, for $i=1,2$;
    \item
    either $\bbf_i$ is exceptional or the map $f_{i,t}$ is ordinary, for $i=1,2$;
    \item 
    The homogeneous lifts $F_{i,t}$ of $f_{i,t}(X)=\frac{a_d(t)X^d+\cdots+a_0(t)}{b_d(t)X^d+\cdots+b_0(t)},$ given by
    $$F_{i,t}(x,y)=(a_d(t)x^d+\cdots+a_0(t)y^d,b_d(t)x^d+\cdots+b_0(t)y^d),$$
    define a function that is smooth in $t$, for $i=1,2$.
    \end{enumerate}
    \end{definition}
   

    The following consequence of Theorem \ref{transcendence} will be crucial in our proof of Theorem \ref{relative bogomolov}.
    
    \begin{proposition}\label{test function}
    Let $\bbphi=(\bbf_1,\bbf_2):\bP_1^2\to \bP_1^2$ be a  split polarized endomorphism of polarization degree $d\ge 2$ and let $\bc\subset \bP_1^2$ be an irreducible curve that is not weakly $\bbphi$-special, both defined over $K$. 
    Then for each compact set $\mathcal{D}\subset B_0\setminus E_{\bbphi, \mathbf{C}}$, there is a continuous function 
    $\psi_{\mathcal{D}}: \pi|^{-1}_{\mathcal{C}}(\mathcal{D})\to \bR$ with the property that
    \begin{align}\label{negativepair}
    \displaystyle\int_{\mathcal{C}_t}\psi_{\mathcal{D}}|_{\mathcal{C}_t}((\deg\pi_2)\pi^{*}_{1,t}\mu_{f_{1,t}}-(\deg\pi_1)\pi^{*}_{2,t}\mu_{f_{2,t}})<0,
    \end{align}
    for each $t\in \mathcal{D}$ such that $\mathcal{C}_t$ is not weakly $\Phi_t$-special. 
    \end{proposition}

    \begin{proof}
    Recall first that both projections $\pi_i$ are dominant since $\bc$ is not weakly $\phi$-special. 
    For $i=1,2$ we let $L_i$ be the line bundle  on $\bc$ associated to $\pi_i^{*}(\infty)$. We denote the specializations at $t\in B_0(\bC)$ by $L_{i,t}$. So the line bundle $L_{i,t}$ corresponds to the divisor $\pi_{i,t}^{*}(\infty)$. 
    We write $d_i:=\deg(\pi_i)$. 
    In what follows we work with the standard absolute value $|\cdot|$ on $\bC$. In particular, $\|(x,y)\|=\max\{|x|,|y|\}$ for $x,y\in\bC$. 
    
    Let $t\in B_0(\bC)\setminus E_{\bbphi,\bc}$ . 
    As recalled in \S\ref{background}, we have an invariant local metric on $L_{i,t}$ (at the fixed archimedean place $|\cdot|$), which we denote by 
    $\overline{L}_{i,t}= (L_{i,t}, \|\cdot\|_{i,t}),$ defined by Zhang \cite{Zhang small} (see also \cite{dmm1}).
    In particular outside of the support of the canonical meromorphic section $\ell_i$ of $L_{i,t}$, with $\mathrm{div}(\ell_i)=\pi^{*}_{i,t}(\infty)$, we have 
    $$-\log\|\ell_i\|_{i,t}(z)=G_{F_{i,t}}(\pi_{i,t}(z),1)=\lim_{n\to\infty} \frac{\log\|F^n_{i,t}(\pi_{i,t}(z),1)\|}{d^n},$$
    where $F_{i,t}$ is the lift from Definition \ref{finite set}. Here we view $\pi_{i,t}(z)$ as a point in $\bC$ (identified with its image in $\bP_1(\bC)$.) Here $G_{F_{i,t}}: \bC^2\setminus\{(0,0)\}\to\bR$ is the escape rate given by 
    \begin{align*}
    G_{F_{i,t}}(x,y)=\lim_{n\to\infty} \frac{\log\|F^n_{i,t}(x,y)\|}{d^n}.
    \end{align*}
    Recall that by construction and since $t\notin E_{\bbphi,\bc}$ we have $c_1(\overline{L}_{i,t})=\pi_{i,t}^{*}\mu_{f_{i,t}}$.
    Since the line bundle $M_t:=d_2L_{1,t}-d_1 L_{2,t}$ has degree $0$, it has an associated flat metric, which we denote by 
    $\overline{M}^{0}_t$ satisfying $c_1(\overline{M}^{0}_t) \equiv 0$ as in  \cite[Definition 5.10]{YZ:index}. 
    Consider the (local) metrized line bundle 
    \begin{align}
    \overline{N_t}:= \overline{d_2 L_{1,t}-d_1 L_{2,t}-M^0_t}=d_2\overline{L}_{1,t}-d_1\overline{L}_{2,t}-\overline{M}^0_t, 
    \end{align}
    and note that by the flatness of $\overline{M}^0_t$ we have
    $$c_{1}(\overline{N}_t)=d_2c_{1}(\overline{L}_{1,t})-d_1c_{1}(\overline{L}_{2,t})=d_2\pi^{*}_{1,t}\mu_{f_{1,t}}-d_1\pi^{*}_{2,t}\mu_{f_{2,t}}.$$
    Since $\overline{N_t}$ is supported on the trivial line bundle in $\mathcal{C}_t$, by the local Hodge-index \cite[Theorem 2.1]{YZ:index} we know that the local pairing satisfies $(\overline{N}_t\cdot \overline{N}_t)\le 0$ with equality if and only if $\overline{N}_t$ is a trivially metrized line bundle. But if the latter happens, then we have  $d_2c_{1}(\overline{L}_{1,t})=d_1c_{1}(\overline{L}_{2,t})$, 
    which by Theorem \ref{transcendence} implies that $\mathcal{C}_t$ is weakly $\Phi_t$-special. 
    Thus for each $t\in B_0(\bC)\setminus E_{\bbphi,\bc}$ such that $\mathcal{C}_t$ is not weakly $\Phi_t$-special, the local pairing is strictly negative; that is
    \begin{align*}
    \displaystyle\int_{\mathcal{C}_t}\left( d_2G_{F_{1,t}}(\pi_{1,t}(z),1)- d_1G_{F_{2,t}}(\pi_{2,t}(z),1)- \log|\theta_t(z)|\right)(d_2d\pi^{*}_{1,t}\mu_{f_{1,t}}(z)-d_1\pi^{*}_{2,t}d\mu_{f_{2,t}}(z))<0,
    \end{align*}
    where $-\log\|\ell_{M_t}\|_{\overline{M}^0_{t}}:=\log|\theta_t|$. Here $\ell_{M_t}$ is the canonical meromorphic section of $M_t$ with $\mathrm{div}(\ell_{M_t})= d_2\pi_{1,t}^{*}(\infty)-d_1\pi_{2,t}^{*}(\infty)$. Moreover, $\theta_t$ is a theta function associated to the curve $\mathcal{C}_t$ and to $\mathrm{div}(\ell_{M_t})$. In detail, $\log |\theta_t|$ may be expressed as $\sum_{P_t \in \mathrm{div}(\ell_{M_t})}\log \|\tilde{\theta}\|\circ\iota_{P_t}$ (taking multiplicities into account), where  $\iota_{P_t}$ is the embedding of $\mathcal{C}_t$ into  its Jacobian $\text{Jac}(\mathcal{C}_t)$ via a smoothly varying base point $P_t$ and $\|\tilde{\theta}\|$ is given as in \cite[p.310]{GriffithsHarris}. Note that, shrinking $B_0$ if needed, we may choose the embedding $\iota:\mathcal{C}\to \mathrm{Jac}(\mathcal{C})$ over $B_0$ so that for all $t\in B_0$, the curve $\mathcal{C}_t$ is not contained in the theta divisor above $t$. 
    		Now let $\mathcal{D}\subset B_0\setminus E_{\bbphi, \bc}$ be a compact set. 
            Choose a smooth function $\chi_{\mathcal{D}}$ so that 
    		$\chi_{\mathcal{D}}\equiv1$ on $\mathcal{D}$ and $\chi_{\mathcal{D}}\equiv 0$ outside a small neighborhood of $\mathcal{D}$ and define
    		\begin{align}
    			\psi_{\mathcal{D}}(t,z)= \chi_{\mathcal{D}}(t)\left( d_2G_{F_{1,t}}(\pi_{1,t}(z),1)- d_1G_{F_{2,t}}(\pi_{2,t}(z),1)- \log|\theta_t(z)|\right),
    		\end{align}
            for $t\in\mathcal{D}$ and $z\in \mathcal{C}_t(\bC)$. 
            We have shown that this function satisfies \eqref{negativepair}. It remains to show that it is continuous. 
            
            Since by Lemma \ref{finite set} the lift $F_{i,t}$ is holomorphic for $t$ in the compact $\mathcal{D}$, from \cite{Hubbard:Papadopol, Fornaess:Sibony} or
            the proof of \cite[Proposition 1.2]{Branner:Hubbard:1}, we know that the escape rate 
            $G_{F_{i,t}}(x,y)$ is continuous as a function of $(t,x,y)\in \mathcal{D}\times \bC^2\setminus\{(0,0)\}$. 
            Moreover, for each $t\in\mathcal{D}$ the function $\pi_{i,t}(z)$ is continuous outside $\pi^{-1}_{i,t}(\infty)$ and since the escape rate scales logarithmically (i.e. $G_{F_{i,t}}(\al x,\al y)=G_{F_{i,t}}(x,y)+\log|\al|$ for each $\al\in\bC\setminus\{0\}$) we infer that 
            $$d_2G_{F_{1,t}}(\pi_{1,t}(z),1)- d_1G_{F_{2,t}}(\pi_{2,t}(z),1)$$
             is continuous as a function of $z\in\mathcal{C}_t$ with logarithmic singularities at $\mathrm{div}(\ell_{M_t})$ (with multiplicities as prescribed there.)
    		The function $\theta_t$ has been chosen exactly so that these logarithmic singularities `cancel out'. 
    		Recall that $\log |\theta_t|= \sum_{P_t \in \mathrm{div}(\ell_{M_t})}\log \|\tilde{\theta}\|\circ\iota_{P_t} $. By construction $\log |\theta_t|$ is smooth on $\mathcal{C}_t$ outside of the support of $\mathrm{div}{\ell_{M_t}}$ and has logarithmic singularities with multiplicities given by $\mathrm{div}{\ell_{M_t}}$. Moreover it depends locally smoothly on the period matrix of $\text{Jac}(\mathcal{C}_t)$ for $t\in B_0\setminus E_{\bbphi, \bbc}$  because of the first condition in Lemma \ref{finite set} and as $M_t$ has degree 0 also globally smoothly on $t$. 
            Therefore $\psi_{\mathcal{D}}$ is continuous. 
            This completes the proof of our proposition.
    \end{proof}
    

	\section{A geometric Bogomolov theorem, a height inequality and applications}\label{functional bogomolov}
	
	In this section we aim to prove the height inequality Theorem \ref{splitineq}. 
	As it turns out, this amounts to establishing the dynamical Bogomolov conjecture for split maps over the function field $K$, which is the content of our next result. 
	\begin{theorem}\label{bogomolov hypothesis}
		Let $\ell\ge 2$ and $\bbphi:\mathbb{P}^{\ell}_1\to \mathbb{P}^{\ell}_1$ be an isotrivial-free split polarized  endomorphism of polarization degree $d\ge 2$ defined over $K$. 
		Let $\mathbf{X}\subset \bP^{\ell}_1$ be an irreducible subvariety defined over $K$ with dimension $\ge 1$. 
		Then the following are equivalent: 
		\begin{enumerate}
			\item 
			$\hat{h}_{\bbphi}(\mathbf{X})=0$;
			\item 
			$\mathbf{X}$ contains a generic\footnote{The sequence $\{a_n\}\subset \mathbf{X}(\Kbar)$ is generic if every infinite subsequence of $\{a_n\}$ is Zariski dense in $\mathbf{X}$.} sequence $\{a_n\}\subset \mathbf{X}(\Kbar)$ of points with $\hat{h}_{\bbphi}(a_n)\to 0$\footnote{We say that $\{a_n\}$ is a $\hat{h}_{\bbphi}$-small sequence.}; 
			\item 
			$\mathbf{X}$ is $\bbphi$-preperiodic. 
		\end{enumerate}
	\end{theorem}

    Moreover, we prove an analogous geometric Bogomolov result in the case that $\mathbf{X}$ is a curve and $\bbphi$ is not isotrivial-free. To state it, we need to define the notion of an isotrivial curve. 
    
    \begin{definition}\label{isotrivial curve}
    We say that $\bbphi = (\mathbf{f}_1, \mathbf{f}_2)$ is isotrivial if both $\bbf_1$ and $\bbf_2$ are isotrivial.  
    If so, let $\mathbf{M}_i$ be Möbius transformations defined over $\overline{K}$ such that $\mathbf{M}_i\circ \mathbf{f}_i\circ \mathbf{M}_i^{-1}$ is defined over $\overline{\mathbb{Q}}$ for $i = 1,2$. 
    A curve $\bc\subset  \mathbb{P}_1\times \mathbb{P}_1$ is $\bbphi$-isotrivial if $(\mathbf{M}_1, \mathbf{M}_2)(\bc)$ is defined over $\overline{\mathbb{Q}}$. 
    \end{definition} 
    
     \noindent Note that this is a field theoretic condition that does not reflect deep  dynamical properties. 
     The analogous geometric Bogomolov statement is then as follows. 
  
    \begin{theorem}\label{mainiso}
        Let  $\bbphi=(\bbf_1,\bbf_2):\mathbb{P}_1^2 \rightarrow \mathbb{P}_1^2$ be a split polarized endomorphism of polarization degree $d\ge 2$ defined over $K$ and assume that at least one of $\bbf_1$ or $\bbf_2$ is isotrivial. 
         Let $\bc \subset \mathbb{P}_1^2$ be an irreducible curve defined over $K$ whose projection to each factor $\mathbb{P}_1$ is dominant. Let $\mathbf{f}:\mathbb{P}_1\to\bP_1$ be an endomorphism of degree $d$ over $K$. 
         Then the following are equivalent. 
        \begin{enumerate}
        \item 
        $\hat{h}_{\bbphi}(\bc)=0$;
        \item 
        $\bc$ contains a generic $\hat{h}_{\bbphi}$-small sequence;
        \item 
        $\bc$ is $\bbphi$-isotrivial and in particular $\bbphi$ is isotrivial;
        \item 
        $\hat{h}_{\bbf\times\bbphi}(\bP_1\times \bc)=0$;
        \item 
        $\hat{h}_{\bbphi\times\bbphi}(\bc\times\bc)=0$. 
        \end{enumerate}
        \end{theorem}

    To place these results in the context of the work of Dimitrov-Gao-Habegger \cite{DGH:pencils,DGH:uml}, K\"uhne \cite{kuhne:uml} and Yuan-Zhang \cite{Yuan:Zhang:new} we make the following remark. 
    
	\begin{remark}\label{nondegenerate}
		Subvarieties with $\hat{h}_{\bbphi}(\mathbf{X})\neq 0$ coincide with non-degenerate subvarieties in the sense of \cite{Yuan:Zhang:new}. 
		This follows by Lemma \ref{height positive}. 
		One can also infer that the two conditions coincide a posteriori by comparing \cite[Lemma 5.22]{Yuan:Zhang:new} and Gauthier-Vigny's result \cite[Theorem B]{GV}.
		Thus, in the language of \cite{Yuan:Zhang:new}, Theorem \ref{bogomolov hypothesis} characterizes the non-degenerate subvarieties. 
	\end{remark}

    An important ingredient of the proof of Theorems \ref{bogomolov hypothesis}, which the authors recently worked out with Wilms is \cite[Theorem 1.1]{note}. It allows to apply the strategy in \cite{dmm2} to give a proof of the dynamical Manin-Mumford conjecture for split endomorphisms over the complex numbers. The latter is important for our proof as we will reduce Theorems \ref{bogomolov hypothesis}, \ref{mainiso} to a  Manin- Mumford type statement over $\mathbb{C}$. For this reduction we rely once again on \cite[Theorem 1.1]{note} as well as on Baker's work \cite{Baker:finiteness}. In the isotrivial setting of Theorem \ref{mainiso}, as we shall see, the mathematical content of the proof is slightly different given that the role of preperiodic curves is replaced by that of isotrivial curves.

    \smallskip 
    
     \noindent Before we begin with our proof we point out that the following holds. 
        \begin{remark} \label{exceptionalnoniso}
        If $\bbf:\bP_1\to\bP_1$ is defined over $K$ and is not isotrivial then $\bbf$ is exceptional if and only if $\bbf$ is a flexible Latt\`es map so that the group $G_{\bbf}$ from Definition \ref{ordinary} is an elliptic curve and $\phi_{\bbf}$ is given by multiplication by an integer. 
        \end{remark}
	
   \noindent We start by investigating the case of hypersurfaces. 
	
	\begin{proposition}\label{bogomolovhyper}
		Let $\ell\ge 2$ and $\bbphi=(\bbf_1,\ldots,\bbf_{\ell}):\mathbb{P}^{\ell}_1\to \mathbb{P}^{\ell}_1$ be an isotrivial-free split polarized endomorphism defined over $K$.
		Let $\mathbf{H}\subset \mathbb{P}^{\ell}_1$ be an irreducible hypersurface defined over $K$ which projects dominantly onto any subset of $\ell-1$ coordinate axis. Assume that $\mathbf{H}$ contains a generic sequence $\{a_n\}\subset \mathbf{H}(\Kbar)$ with $\hat{h}_{\bbphi}(a_n)\to 0$ as $n\to\infty$. 
		Then 
			$\mathbf{H}$ is $\bbphi$-preperiodic. Moreover, either $\ell = 2$ or all maps $\bbf_1, \ldots, \bbf_{\ell}$ are exceptional. 
	\end{proposition}
	
	\begin{proof}
    Let $\mathbf{H}\subset \bP_1^{\ell}$ be as in the statement of the proposition. 
         We write $\pi_i: \bP_1^{\ell}\to\bP_1$ to denote the projection to the $i$-th coordinate for $i\in\{1,\ldots,\ell\}$. 
         By  \cite[Theorem 1.1]{note} our assumptions imply that there exist a closed subset $\mathbf{E}\subsetneq \mathbf{H}$ such that 
         \begin{align}\label{coincidence}
         \{\bbx\in \mathbf{H}\setminus \mathbf{E}(\Kbar): \hat{h}_{\bbf_i}(\pi_i(\bbx))=0, i\in\{1,\ldots,\ell-1\}\}=\{\bbx\in \mathbf{H}\setminus \mathbf{E}(\Kbar): \hat{h}_{\bbphi}(\bbx)=0\}.
         \end{align}
         Since each $\bbf_{i}$ is not isotrivial, \cite[Corollary 1.8]{Baker:finiteness} implies that $\hat{h}_{\bbf_i}(\mathbf{a})=0$ if and only if $\mathbf{a}\in\mathrm{Prep}(\bbf_i)$ for $i\in \{1,\ldots,\ell\}$. 
         Letting $p_{\ell}:\mathbf{H}\to \bP_1^{\ell-1}$ denote the projection to the first $\ell-1$ copies of $\bP_1$, and recalling that $\hat{h}_{\bbphi}=\displaystyle\sum_{i=1}^{\ell} \hat{h}_{\bbf_i}$, equation \eqref{coincidence} can therefore be rephrased as 
                 \begin{align}\label{coincidence2}
                  p_{\ell}^{-1}(\mathrm{Prep}(\bbf_1)\times \cdots\times \mathrm{Prep}(\bbf_{\ell-1}))\setminus \mathbf{E}(\Kbar)=\mathrm{Prep}(\bbphi)\cap (\mathbf{H}\setminus \mathbf{E}(\bar{K})). 
                          \end{align}
         Henceforth we embed $\Kbar\hookrightarrow\bC$ and abusing the notation slightly we view $\bbf_1,\ldots,\bbf_{\ell}$ and $\mathbf{H}$ as being defined over $\Kbar \subset\bC$. 
         Since each $\bbf_i$ has infinitely many preperiodic points, $\mathrm{Prep}(\bbf_1)\times \cdots\times \mathrm{Prep}(\bbf_{\ell-1})$ is Zariski dense in $\bP_1^{\ell-1}$. Since also $p_{\ell}$ is dominant and $\mathbf{E}$ is a strict closed subset of $\mathbf{H}$, equation \eqref{coincidence2} implies that $\mathbf{H}$ contains a Zariski dense set of complex  points in $\mathrm{Prep}(\bbphi)$. 
         Now if $\ell>2$, then by \cite[Theorem 2.2]{dmm2} and in view of Remark \ref{exceptionalnoniso} we infer that all components $\bbf_i$ of $\bbphi$ are Lattès maps corresponding to integer-multiplications on elliptic curves. By \cite{Ra2} (see also \cite{Gubler:split} and \cite[Theorem 1.4]{Habegger:special} for special cases) it follows that $\mathbf{H}$ is $\bbphi$-preperiodic. 
        If on the other hand $\ell=2$, then the conclusion follows by \cite[Theorem 1.1]{dmm1} if one of the $\bbf_i$ is ordinary or by \cite{Ra2}  otherwise.
	\end{proof}
	
	\noindent We are now ready to prove Theorem \ref{bogomolov hypothesis}. 
	
	\noindent\emph{Proof of Theorem \ref{bogomolov hypothesis}:}
	That $(1)$ and $(2)$ are equivalent follows by Gubler's fundamental inequalities as in Lemma \ref{zhang ineq}. In detail, first notice that the height $\hat{h}_{\bbphi}$ takes non-negative values when evaluated at points and as a result all the successive minima associated to $\overline{L}_{\bbphi}$ are non-negative. It then follows by \cite[Lemma 4.1]{Gubler:split} that $\hat{h}_{\bbphi}(\mathbf{X})=0$ if and only if the first successive minimum of $h_{\overline{L}_{\mathbf{\Phi}}|_{\mathbf{X}}}$ is equal to zero or equivalently if $(2)$ holds. 
	That $(3)$ implies $(1)$ is well-known, see e.g \cite[Theorem 3.2 (iv)]{Xander}. 
	It remains to see that $(2)$ implies $(3)$. 
	But as in \cite[\S{2},Proposition 2.1]{dmm2} (see also \cite{mm:new}), the situation reduces to that of Proposition \ref{bogomolovhyper}. 
	This completes our proof. 
	\qed

  \smallskip
  
  \noindent To prove Theorem \ref{mainiso} we also need the following result that is implicitly given in \cite{Baker:finiteness} and builds on some ideas that first appeared in \cite{Baker:lowerbound}. 
   \begin{theorem}\label{baker} \cite{Baker:finiteness} 
       Let $\mathbf{f}\in K(z)$ be a non-isotrivial rational map of degree at least $2$ and let $K'$ be a finite extension of $K$. Then there exists a constant $C>0$ depending only on $\bbf$ and $K$ such that 
       $$\#\{P\in \bP_1(K')~:~\hat{h}_{\bbf}(P)=0\}\le C[K':K]\log[K':K].$$ 
       In particular, there is a constant $C'$ depending only on $K$ and $\bbf$ such that if a finite extension $L$ of $K$ contains $N$ distinct $\bbf$-preperiodic points, then 
       $$[L:K]\ge C' \frac{N}{\log N}.$$
   \end{theorem}

\begin{remark}\label{nftoff}
Theorem \ref{baker} is established in \cite[Theorem 1.14]{Baker:lowerbound} in the case of number fields. The proof relies on the existence of an archimedean place to give a  repulsion of points in parts of $\bP^1(\bC)$; c.f. \cite[Lemma 3.1]{Baker:lowerbound}. In view of \cite[Theorem 1.11]{Baker:finiteness} one can get the same repulsion property in a place of genuinely bad reduction of a rational function $f$ defined over function field. Moreover, non-isotrivial maps have a place of bad reduction by \cite[Theorem 1.9]{Baker:finiteness}. Therefore, Theorem \ref{baker} can be proved following the same lines as its number field counterpart. 
\end{remark}

   \begin{proof}[Proof of Theorem \ref{baker}]
  We will use the notation of \cite{Baker:finiteness}. 
  Let $K'$ be a finite extension of $K$ and let $D:=[K':K]$. By the product formula the canonical height decomposes into local contributions
\begin{align}\label{decomposition}
\hat{h}_{\bbf}(x)+\hat{h}_{\bbf}(y)=\frac{1}{D}\sum_{v\in M_{K'}}g_{\bbf,v}(x,y),
\end{align}
for two distinct $x,y\in \mathbb{P}_1(K')$.
 Here $g_{\bbf,v}$ is the two variable dynamical Green’s function as in \cite[\S 3]{Baker:finiteness}. 
 Since $\mathbf{f}\in K(z)$ is not isotrivial, there exists a place $v_0\in M_K$ where it has genuinely bad reduction by \cite[Theorem 1.9]{Baker:finiteness}. 
 By \cite[Theorem 1.11]{Baker:finiteness}, there is a finite covering $V_1,...,V_s$ of $\mathbb{P}_1(\bC_{v_0})$ and a constant $C_1>0$ such that 
 \begin{align}\label{repulsion}
 g_{\mathbf{f},v_0}(x,y)\ge C_1, ~x,y\in V_i, ~i=1,\ldots, s. 
 \end{align}
 Let $x_1,\ldots,x_N\in\mathbb{P}_1(K')$ be $\bbf$-preperiodic points for $N\in\bN$ and let $M=[\frac{N-1}{s}]+1$. By the pigeonhole principle we have that at least $M$ of our points $x_1,\ldots,x_N$ belong in the same $V_i$. After relabelling we may assume that $x_1,\ldots,x_M\in V_i$. 
 In what follows, we define 
 \begin{align}
g(x_1,\ldots,x_M):=\frac{1}{D} \sum_{1\le i\neq j\le M}\sum_{v\in M_{K'}}g_{\mathbf{f},v}(x_i,x_j).
\end{align}
Since the $x_i$ are preperiodic, by \eqref{decomposition} we infer that 
\begin{align}\label{zero}
g(x_1,\ldots,x_M)=0.
\end{align}
On the other hand, by \cite[Theorem 3.15]{Baker:finiteness} we have that $g_{\mathbf{f},v}\ge 0$ at all places $v\in M_{K'}$ where $\mathbf{f}$ has potential good reduction, so \cite[Theorem 1.14]{Baker:lowerbound}, our assumption that $x_1,\ldots, x_M\in V_i$ and \eqref{repulsion} imply that 
\begin{align}\label{lowerbound}
g(x_1,\ldots,x_M)\ge \frac{C_1}{D}M^2-C_2M \log M,
\end{align}
for positive constants $C_1,C_2$ that only depend on $\mathbf{f}$ and $K$. 
Combining \eqref{zero} and \eqref{lowerbound} we infer that 
$$\frac{C_1}{2D} M\le C_2\log M,$$
 and since $N\le Ms+1$ by our definition of $M$, we get that $N\le CD\log D$ as claimed. It also follows directly that $[L:K]\ge C'\frac{N}{\log N}$ for any extension $L$ of $K$ that contains $N$ preperiodic points. 
   \end{proof}

   \begin{corollary}\label{prep_extension}
 Let $\mathbf{f}\in K(z)$ be a non-isotrivial rational map of degree at least $2$ and let $D\in \mathbb{Z}$. Then the set 
 $$\{P\in \mathrm{Prep}(\bbf)~:~[K(P):K]\le D\}$$
 is finite. 
   \end{corollary}

   \begin{proof}
Let $P\in \mathrm{Prep}(\bbf)$ such that $[K(P):K]\le D$. We have that $\mathcal{O}_{\bbf}(P)=\{\bbf^i(P)~:~i\in\mathbb{Z}_{\ge 0}\}\subset K(P)$ and by Theorem \ref{baker} and our assumption that $[K(P):K]\le D$ we infer that the size of the orbit of $P$ is bounded by a constant that depends only on $D, \bbf, K$. 
This gives bounds for the smallest non-negative integer $m$ so that $\bbf^m(P)$ is periodic, and for the period of $\bbf^m(P)$, that depend only on $D,\bbf, K$. Therefore there are only finitely many such preperiodic points. 
   \end{proof}
 \noindent  We can now prove Theorem \ref{mainiso}. 
 
 \noindent  \emph{Proof of Theorem \ref{mainiso}:}
Since $\hat{h}_{\bbphi}(\bc)=\hat{h}_{\bbf\times\bbphi}(\bP_1\times\bc)=\frac{1}{2}\hat{h}_{\bbphi\times\bbphi}(\bc\times\bc)$, it is clear that $(1),(4)$ and $(5)$ are equivalent. That $(1)$ and $(2)$ are equivalent follows by Gubler's inequality in Lemma \ref{zhang ineq} as in the proof of Theorem \ref{bogomolov hypothesis}. 
 
 To see that $(3)$ implies $(2)$ note that if $\mathbf{M}_1,\mathbf{M}_2$ are M\"obius maps such that $(\mathbf{M}_1,\mathbf{M}_2)\circ \bbphi\circ(\mathbf{M}_1^{-1},\mathbf{M}_2^{-1})$ is defined over $\Qbar$ and $(\mathbf{M}_1, \mathbf{M}_2)(\bc)$ is defined over $\overline{\mathbb{Q}}$, then $\bc$ contains infinitely many points $P$ with $\hat{h}_{\bbphi}(P)=0$. 
 
 It remains to show that $(2)$ implies $(3)$. 
   	Let $\bc\subset\bP_1^2$ be an irreducible curve over $K$ such that the projections $\pi_1,\pi_2:\bc\to \mathbb{P}_1$ to both coordinates are dominant. Assume that $\bc$ contains a generic $\hat{h}_{\bbphi}$-small sequence. 
     By  \cite[Theorem 1.1]{note} we infer that  
                \begin{align}\label{coincidencecurves}
                \{\bbx\in \mathbf{C}(\Kbar): \hat{h}_{\bbf_1}(\pi_1(\bbx))=0\}=\{\bbx\in \mathbf{C}(\Kbar): \hat{h}_{\bbphi}(\bbx)=0\}.
                \end{align}
       Since $\bbf_1$ has infinitely many preperiodic points we therefore deduce  that $\bc$ contains a Zariski-dense set of points $\bbx$ of height $\hat{h}_{\bbphi}(\bbx)=0$. If $\bbphi$ is isotrivial and $(\mathbf{M}_1, \mathbf{M}_2)$ is a pair of Möbius transformations such that $(\mathbf{M}_1,\mathbf{M}_2)\circ \bbphi\circ(\mathbf{M}_1^{-1},\mathbf{M}_2^{-1})$ is defined over $\overline{\Q}$, then in view of \cite[Theorem 1.1]{D:stableheight} this implies that $(\mathbf{M}_1,\mathbf{M}_2)(\bc)$ contains infinitely many points in $\bP_1^2(\overline{\mathbb{Q}}$). Therefore, $(\mathbf{M}_1, \mathbf{M}_2)(\bc)$  is defined over $\overline{\mathbb{Q}}$ and $\bc$ is $\bbphi$-isotrivial as claimed. 
       Assume now that exactly one of $\mathbf{f}_1$ or $\mathbf{f_2}$ is isotrivial to end in a contradiction. Without loss of generality we may assume that $\mathbf{f}_1$ is isotrivial. We let $\mathbf{M}_1$ be a M\"obius transformation such that $\mathbf{M}_1\circ \mathbf{f}_1\circ\mathbf{M}^{-1}_1$ is defined over $\Qbar$. 
           In view of \cite[Corollary 1.8]{Baker:finiteness}, equation \eqref{coincidencecurves} implies that the curve $(\mathbf{M}_1, \text{id})(\bc)$ contains infinitely many points $(x,y)$ such that the first coordinate is in $\mathbb{P}_1(\overline{\Q})$ (and thus in particular in $\mathbb{P}_1(K)$) and the second coordinate is an $\bf{f}_2$-preperiodic point. Recalling also that $\mathbf{C}$ projects dominantly to both coordinates and since the points $(x,y)$ satisfy the algebraic equation given by $(\mathbf{M}_1, \text{id})(\bc)$, we infer that the degree $[K(y):K]$ is bounded. In other words, we have infinitely many preperiodic points $y$ of $\mathbf{f}_2$ with bounded degree $[K(y):K]$. This is a contradiction to our assumption that $\mathbf{f}_2$ is non-isotrivial. Indeed, by Corollary \ref{prep_extension} we know that the non-isotrivial $\mathbf{f}_2$ has only finitely many preperiodic points of bounded degree. 
       This completes the proof of this theorem. 
\qed

    \smallskip
   
	\noindent We can now deduce Theorem \ref{splitineq}. 

	\smallskip
    
	\noindent\emph{Proof of Theorem \ref{splitineq}:} 
	If $\mathbf{X}$ has dimension equal to $0$ (and is therefore a point) then the result follows combining Call-Silverman's result \cite[Theorem 4.1]{Call:Silverman} and Baker's \cite[Theorem 1.6]{Baker:finiteness}, since $\bbphi$ is isotrivial-free. Indeed, either $\hat{h}_{\bbphi}(\mathbf{X})=0$ and since $\bbphi$ is isotrivial-free by \cite[Theorem 1.6]{Baker:finiteness} we have that $\mathbf{X}$ is $\bbphi-$preperiodic as in case (1) of the Theorem, or $\hat{h}_{\bbphi}(\mathbf{X})>0$ and by Call-Silverman's \cite[Theorem 4.1]{Call:Silverman} we get  a positive constant $c_1$ as in case (2) of the Theorem.
 We thus let $\ell\ge 2$ and $\mathbf{X}\subset\mathbb{P}^{\ell}_1$ be a positive dimensional irreducible subvariety over $K$. 
	First note that by Theorem \ref{bogomolov hypothesis} we have that $\hat{h}_{\bbphi}(\mathbf{X})=0$ if and only if $\mathbf{X}$ is $\bbphi$-preperiodic 
	or equivalently $\mathcal{X}^{\bbphi,\star}=\emptyset$. 
	Assume now that $\mathbf{X}$ is not $\bbphi$-preperiodic, so that $\mathcal{X}^{\bbphi,\star}\neq\emptyset$.  
	We need to show that $\mathcal{X}^{\bbphi,\star}\neq\emptyset$ is Zariski open and dense in $\mathcal{X}$. 
	To this end, it suffices to show that there are at most finitely many $\mathcal{Z}\subset \mathcal{X}$ that are irreducible components of $\Phi$-special subvarieties projecting dominantly onto $B$ and are maximal with this property. 
	Equivalently, we have to show that $\mathbf{X}$ (which is not $\bbphi$-preperiodic) contains at most finitely many maximal $\bbphi$-preperiodic subvarieties. 
	This in turn follows from our geometric Bogomolov Theorem \ref{bogomolov hypothesis} (a Manin-Mumford version would suffice). 
	Now since $\mathbf{X}$ is not $\bbphi$-preperiodic, by Theorem \ref{bogomolov hypothesis} we infer that  $\hat{h}_{\bbphi}(\mathbf{X})>0$. Thus by Theorem \ref{assume bog} or by \cite[Theorem 6.2.2]{Yuan:Zhang:new} (recall  Remark \ref{nondegenerate}) we deduce that there is a proper Zariski closed $\mathcal{V}\subset \mathcal{X}$ so that our desired inequality \eqref{lower bound} holds for all $P\in (\mathcal{X}^{\bbphi,\star}\setminus\mathcal{V})(\Qbar)$. 
	We will show that our inequality holds for all $\Qbar$-points of $\mathcal{X}^{\bbphi,\star}$ by induction of the dimension on $\mathcal{X}$. If $\mathcal{X}$ has dimension equal to $1$, then the exceptional set $\mathcal{V}$ is merely a finite collection of points which we can incorporate by adjusting the constants in the inequality. Thus the base case follows. Now assume that $\mathcal{X}$ has dimension at least $2$. 
	We decompose $\mathcal{V}$ into finitely many irreducible components $\mathcal{V}_i$ for $i=1,\ldots,s$. Since the inequality \eqref{lower bound} holds on each vertical fiber - adjusting the constants if neccessary, we assume henceforth that all components $\mathcal{V}_i$ are horizontal over $B_0$.
	So it suffices to show that the inequality holds at all points in $\mathcal{V}_i\cap \mathcal{X}^{\bbphi,\star}$ for $i=1,\ldots,s$. 
	Notice that $\mathcal{V}_i$ has codimension at least one in $\mathcal{X}$, so we may induct on its dimension and apply our theorem to each $\mathcal{V}_i$. Adjusting the constants, we conclude that the inequality \eqref{lower bound} holds for all $\Qbar$-points in 
	$\mathcal{X}^{\bbphi,\star}\subset(\mathcal{X}^{\bbphi,\star}\setminus\mathcal{V})\cup \mathcal{V}^{\bbphi,\star}_1\cup\cdots\cup \mathcal{V}^{\bbphi,\star}_s$. 
	This completes our proof. 
	\qed

\smallskip 

\noindent In fact, arguing exactly as in the proof above and using Theorem \ref{mainiso} in place of Theorem \ref{bogomolov hypothesis} we infer that the following holds. 
\begin{theorem}\label{splitineqiso}
	Let $\bc \subset \mathbb{P}_1^2$ be a curve over $K$  and $\bbphi:\mathbb{P}_1^2 \rightarrow \mathbb{P}_1^2$ be a polarized split  map $\bbphi = (\mathbf{f}_1,\mathbf{f}_2)$ such that at least one of $\mathbf{f}_1, \mathbf{f}_2$ is  isotrivial. If $\bbphi$ is isotrivial assume that $\bc$ is not $\bbphi$-isotrivial. Then  $\mathcal{C}^{\bbphi, \star}$ is Zariski open in $\mathcal{C}$ and there are constants $c_1=c_1(\bbphi,\bc)>0$ and $c_2=c_2(\bbphi,\bc)$ such that 
	\begin{align*}
		\hat{h}_{\Phi}(P) \ge c_{1}h_{\mathcal{N}}(\pi(P))-c_{2}, \text{ for all }P\in\mathcal{C}^{\bbphi, \star}(\Qbar).
	\end{align*}
\end{theorem}

\subsection{Applications of Theorem \ref{splitineq}: arithmetic equidistribution}
	
	Towards the proof of Theorems \ref{relative bogomolov} and \ref{umb} we will need an arithmetic equidistribution result. To this end, we rely on recent work of Yuan and Zhang \cite{Yuan:Zhang:new}, who generalized K\"uhne's equidistribution theorem \cite[Theorem 1]{kuhne:uml}, by building a theory of adelic metrized line bundles on essentially quasiprojective varieties. 
    Using this theory Yuan \cite{yuan} has recently provided a new proof of \cite[Theorem 2 and Theorem 3]{kuhne:uml}. 
	
    Let $\bbphi$ and $\mathbf{X}$ be as in \S\ref{functional bogomolov}. 
	Associated to the pair $(\bbphi,\mathbf{X})$ we have a pair 
	$(\Phi,\mathcal{X})$ defined over a number field $\mathcal{K}$, where $\Phi: B_0\times \mathbb{P}^{\ell}_1\to B_0\times \mathbb{P}^{\ell}_1$ and $\mathcal{X}\subset B_0\times\mathbb{P}^{\ell}_1$. 
    In what follows we fix an embedding $\mathcal{K}\subset \bC$ and as a result an archimedean place $|\cdot|_{\infty}$ of $\mathcal{K}$ extending the classic absolute value in $\bQ$. 
    By \cite[\S 6.1.1, Theorem 6.1]{Yuan:Zhang:new} we can associate an invariant adelic line bundle $\overline{L}_{\Phi}$ to the endomorphism $\Phi$ which has the property that $\Phi^{*} \overline{L}_{\Phi}\iso \overline{L}_{\Phi}^{\otimes d}.$ 
	We further have a canonical measure $\mu_{\overline{L}_{\Phi}}|_{\mathcal{X}}$ defined in  $\mathcal{X}(\mathbb{C})$ as in \cite[\S3.6.6]{Yuan:Zhang:new}.
	
	\begin{definition}\label{generic small}
		We say that $Z:=\{x_n\}_n\subset \mathcal{X}(\Qbar)$ is a \emph{$(\Phi,\mathcal{X})$-generic small sequence} if every infinite subsequence of $Z$ is Zariski dense in $\mathcal{X}$ and moreover $\hat{h}_{\Phi}(x_n)\to 0$ as $n\to \infty$. 
	\end{definition}
	We can now state the equidistribution theorem that will be crucial for our proof.
	\begin{proposition}\label{equidistribution}
        Let $\ell\in\bZ_{\ge 1}$. 
		Let $\bbphi=(\bbf_1,\ldots,\bbf_{\ell}): \mathbb{P}^{\ell}_1\to \mathbb{P}^{\ell}_1$ be a split polarized endomorphism of polarization degree at least $2$ and let  $\mathbf{X}\subset \mathbb{P}^{\ell}_1$ be an irreducible subvariety, both defined over $K$. 
        Assume that one of the following holds. 
        \begin{enumerate}
        \item $\bbphi$ is isotrivial-free and $\mathbf{X}$ is not $\bbphi$-preperiodic; or 
        \item $\ell=2$, exactly one of $\bbf_1$ or $\bbf_2$ is isotrivial and $\mathbf{X}$ is a curve whose projection to each factor $\bP_1$ is dominant; or 
        \item $\ell=2$, $\bbphi$ is isotrivial and $\mathbf{X}$ is a non $\bbphi$-isotrivial curve that projects dominantly to each factor $\bP_1$. 
        \end{enumerate}
		Let $\{x_n\}_n\subset \mathcal{X}(\Qbar)$ be a $(\Phi,\mathcal{X})$-generic small sequence. 
		Then the Galois orbit of $\{x_n\}_n$ is equidistributed over $\mathcal{X}(\bC)$ with respect to the canonical measure $\mu_{\overline{L}_{\Phi}|\mathcal{X}}$. 
	\end{proposition}
	\begin{proof}
		Recalling Remark \ref{nondegenerate}, this an immediate consequence of our geometric Bogomolov results \ref{bogomolov hypothesis} and \ref{mainiso} coupled with \cite[Theorem 1.4.2]{Yuan:Zhang:new}.
	\end{proof}
    
  It is worth pointing out that the aforementioned equidistribution result \cite[Theorem 1.4.2]{Yuan:Zhang:new} yields a similar statement for each place $v\in M_{\mathcal{K}}$. We shall only need their result at the fixed archimedean place. K\"uhne has pointed out that his proof of \cite[Theorem 1]{kuhne:uml} adapts to this dynamical setting and it would be sufficient for us. 
  Gauthier \cite{Gauthier:equi} has also recently established a related equidistribution result. 

We conclude this section by observing the following simple but crucial for us consequence of Theorem \ref{splitineq} and (a very special case of) Proposition \ref{equidistribution}. 

\begin{proposition}\label{compactum}
Under the same assumptions as in Proposition \ref{equidistribution}, assume further that there exists a $(\Phi,\mathcal{X})$-generic small sequence. 
Let $F\subset B_0(\Qbar)$ be a finite set.  
Then there is a compact set $\mathcal{D}=\mathcal{D}(F)\subset B_0\setminus F$ such that 
$\mu_{\overline{L}_{\Phi}|\mathcal{X}}(\pi|_{\mathcal{X}}^{-1}(\mathcal{D}))>0.$ 
	\end{proposition}	
		
	\begin{proof}
	Let $F\subset B_0(\bC)$ be a finite set. 
	Let  $(t_n,z_n)$ be a generic $(\Phi, \mathcal{X})$-small sequence. 
	By Theorems \ref{splitineq} and \ref{splitineqiso} we know that $\mathcal{X}^{\bbphi,\star}$ is Zariski dense in $\mathcal{X}$, so we may assume passing to a subsequence that $(t_n,z_n)\in \mathcal{X}^{\bbphi,\star}(\Qbar)$ for all $n\in\bN$. Let $h$ be a height on $B$ associated to an ample divisor. 
	Invoking Theorems \ref{splitineq} and \ref{splitineqiso} we infer that $h(t_n)\le M$ for a constant $M$ depending only on $(\bbphi, \mathbf{X})$. 
   This in turn prevents a positive proportion of Galois conjugates of $t_n$ from accumulating near $F$; see Masser-Zannier's \cite[Lemma 8.2]{Masser:Zannier:2}. 
    More precisely, there is a compact set $\mathcal{D}=\mathcal{D}(F,\bbphi,\mathbf{X})\subset B_0(\bC)\setminus F$ and a constant $a>0$ depending only on $\bbphi$, $\mathbf{X}$ and $F$ such that
		\begin{align}\label{many conjugates}
			\frac{1}{[\mathcal{K}(t_n):\mathcal{K}]}	\displaystyle\sum_{\sigma\in\Gal(\mathcal{K}(t_n)/\mathcal{K})}\chi_{\mathcal{D}}(\sigma(t_n))\ge a,
		\end{align}
	for all $n\in\bN$, where $\chi_{\mathcal{D}}$ denotes the charasteristic function of $\mathcal{D}$. 
	Now we can find a decreasing sequence of continuous non-negative  functions $\chi_{\mathcal{D},\epsilon}: B_0\times \mathbb{P}^{\ell}_1\to \bR_{\ge 0}$ for $\epsilon>0$ such that $\chi_{\mathcal{D},\epsilon}(Z)=\chi_{D}(\pi(Z))=1$, for $Z$ with $\pi(Z)\in \mathcal{D}$ and $\chi_{\mathcal{D},\epsilon}(Z)=0$ for $Z$ with $\pi(Z)$ away from a compact $\epsilon$-neighborhood of $\mathcal{D}$. In particular $\chi_{\mathcal{D},\epsilon}(Z)\ge \chi_{\mathcal{D}}(\pi(Z))$ for all $Z$. Applying Proposition \ref{equidistribution} with test function $\chi_{\mathcal{D},\epsilon}$ along our sequence of small points we infer
		\begin{align*}
			\frac{1}{[\mathcal{K}(t_n,z_n):\mathcal{K}]}	\displaystyle\sum_{\sigma\in\Gal(\mathcal{K}(t_n)/\mathcal{K})}\sum_{\rho\in \Gal(\mathcal{K}(t_n,z_n)/\mathcal{K}): \rho|\mathcal{K}(t_n)=\sigma }\chi_{\mathcal{D},\epsilon}(\sigma(t_n),\rho(z_n))\to \int \chi_{\mathcal{D},\epsilon} d\mu_{\overline{L}_{\Phi}|_{\mathcal{C}}}.
		\end{align*}
		The definition of $\chi_{D,\epsilon}$ along with inequality \eqref{many conjugates} then allow us to conclude that 
		$$ \int \chi_{\mathcal{D},\epsilon} d\mu_{\overline{L}_{\Phi}|_{\mathcal{C}}}\ge a,$$
		for each $\epsilon$. Letting $\epsilon\to 0$ the lemma follows by monotone convergence.
	\end{proof}


\section{Instances of relative Bogomolov for split maps}\label{fighting the currents}

In this section we prove Theorem \ref{relative bogomolov} and subsequently deduce Theorem \ref{umb} as well as Corollary \ref{small curves}. 
We will need the notion of a fiber-wise Green current associated to a family of polarized endomorphisms, which we shall first introduce. For the basics on currents we refer the reader to \cite{DS}. Our presentation in terms of currents is inspired by the slick presentation in \cite{GV}.

Let $\Lambda$ be a smooth quasi-projective complex variety. Let $\pi: \mathcal{A}\to \Lambda$ be a family of irreducible complex projective varieties of dimension $k\ge 1$, so that $\mathcal{A}_{\la}=\pi^{-1}(\{\la\})$ is an irreducible complex variety of dimension $k$ for each $\la\in \Lambda$. 
Let $f:\mathcal{A}\to \mathcal{A}$ be complex family of endomorphisms, i.e. $f$ is analytic and for each $\lambda\in \Lambda$ the map  $f_{\la}:\mathcal{A}_{\la}\to \mathcal{A}_{\la}$ is a morphism. Assume that $f^{*}\mathcal{L}\iso \mathcal{L}^{\otimes d}$ for a line bundle $\mathcal{L}$ on $\mathcal{A}$, inducing an isomorphism $f_{\la}^{*}\mathcal{L}_{\la}\iso\mathcal{L}^{\otimes d}_{\la}$ for each $\la\in \Lambda$. Here $\mathcal{L}_{\la}=\mathcal{L}|_{\mathcal{A}_{\la}}$. 
We encode the above data by $(\Lambda, \mathcal{A},f,\mathcal{L})$. 
We let $\widehat{\omega}$ be a continuous positive $(1,1)$-form on $\mathcal{A}$ cohomologous to a multiple of $\mathcal{L}$ such that $\omega_{\la}:=\widehat{\omega}|_{\mathcal{A}_{\la}}$ is a K\"ahler form on $(\mathcal{A}_{\la})_{\mathrm{reg}}$. 
We assume that $\widehat{\omega}$ is \textbf{normalized} so that 
\begin{align}\label{normalize omega}
\int_{\mathcal{A}_{\la}}\omega^k_{\la}=1,
\end{align}
for each $\la\in\Lambda$. 
The sequence $d^{-n}(f^n)^{*}(\widehat{\omega})$ converges in the sense of currents to a closed positive $(1,1)$-current $\widehat{T}_{f}$ with continuous potentials on $\mathcal{A}$. 
We refer to $\widehat{T}_f$ as the \emph{fiber-wise Green current} of $f$.  
Moreover, for all $\la\in \Lambda$ the slice $\widehat{T}_f|_{\mathcal{A}_{\la}}$ is well-defined and we have $f^{*}_{\la}\widehat{T}_f|_{\mathcal{A}_{\la}}=d\widehat{T}_f|_{\mathcal{A}_{\la}}$. For these facts we refer the reader to  \cite[Proposition 9]{GV}. 

	\begin{remark}\label{measures as currents}
		In the setting of Proposition \ref{equidistribution}, we have  $\mu_{\overline{L}_{\Phi}|\mathcal{X}}= \frac{1}{v} \widehat{T}^{\wedge \dim \mathcal{X}}_{\Phi}\wedge [\mathcal{X}]$ for a constant $v>0$. This transpires from the construction \cite[\S 3.6.6]{Yuan:Zhang:new} of $\mu_{\overline{L}_{\Phi}|\mathcal{X}}$ via a limiting process. By \cite[Theorem B]{GV} we know that $v=\hat{h}_{\bbphi}(\mathbf{X})$ and $\mu_{\overline{L}_{\Phi}|\mathcal{X}}$ is a probability measure; see also \cite[Lemma 5.4.4]{Yuan:Zhang:new}. This latter fact is not used in our proof of Theorem \ref{relative bogomolov}.
	\end{remark}

{\sf In what follows we work in the setting of Theorem \ref{relative bogomolov}.}
So $\bbphi=(\bbf_1,\bbf_2):\bP_1^2\to \bP_1^2$ and $\bbpsi: \bP_1^{\ell}\to \bP_1^{\ell}$ for $\ell\ge 1$ are isotrivial-free split polarizable endomorphisms both of polarization degree $d\ge 2$, $\bc \subset\bP_1^2$ is an irreducible curve that is not weakly $\bbphi$-special and $\mathbf{X}\subset\bP_1^{\ell}$ is an irreducible subvariety that is not $\bbpsi$-preperiodic all defined over $K$. These induce endomorphisms of quasiprojective varieties $\Phi: B_0\times \bP_1^2\to B_0\times \bP_1^2$, $f_i:B_0\times \bP_1\to B_0\times \bP_1$ for $i=1,2$ and $\Psi: B_0\times \bP_1^{\ell}\to B_0\times \bP_1^{\ell}$, as well as subvarieties $\mathcal{C}\subset B_0\times \bP_1^2$ and $\mathcal{X}\subset B_0\times \bP_1^{\ell}$ defined over $\Qbar$. 
For $i=1,2$ we write 
\begin{align*}
f_i\times_{B_0}\Psi: B_0\times \bP_1^{\ell+1}&\to  B_0\times \bP_1^{\ell+1}\\
(t,z_1,z_2\ldots,z_{\ell+1})&\mapsto (t, f_{i,t}(z_1),\Psi_t(z_2,\ldots,z_{\ell+1})).
\end{align*}
Finally, we fix several projection maps as in the following diagram. 
	\begin{equation}\label{diagram}
		\begin{tikzcd}
			\mathcal{C}\times_{B_0} \mathcal{X}\arrow{d}{\pi_{\mathcal{C}}}  \arrow[r, yshift=0.7ex, "\tilde{\pi}_1"] \arrow[r, yshift=-0.7ex, swap , "\tilde{\pi}_2"] & \bP_1\times \mathcal{X} \arrow{d}{p}\arrow{r}{q} & \mathcal{X}  \arrow{d}{q_B}\\
			\mathcal{C}  \arrow[r, yshift=0.7ex, "\pi_1"] \arrow[r, yshift=-0.7ex, swap, "\pi_2"]	& B_0\times \bP_1 \arrow{r}[swap]{p_B}& B_0
		\end{tikzcd}
	\end{equation}
\noindent Here $\pi_i$ and $\tilde{\pi}_i$ denote the projections to the $i$-th copy of $\bP_1$ for $i=1,2$. 

Our first step to proving Theorem \ref{relative bogomolov} is to apply the equidistribution result in Proposition \ref{equidistribution} (in fact $3$ times) and deduce the following equality of measures. 

\begin{lemma}\label{equidistributiontwice}
Under the assumptions in Theorem \ref{relative bogomolov} assume further that there is a $(\Phi\times_{B_0}\Psi, \mathcal{C}\times_{B_0}\mathcal{X})$-generic small sequence. 
Then there exists $\al\in \bR_{>0}$ such that 
\begin{align}\label{push and pull}
\pi_1^{*}p_{*}(\widehat{T}^{\wedge (1+\dim \mathcal{X})}_{f_1\times _{B_0}\Psi}\wedge [\bP_1\times\mathcal{X}]) =\al  \pi_2^{*}p_{*}(\widehat{T}^{\wedge (1+\dim \mathcal{X})}_{f_2\times _{B_0}\Psi}\wedge [\bP_1\times\mathcal{X}]).
\end{align}
\end{lemma}

\begin{proof}
	Let $z:=\{z_n\}_{n\in\bN}$ be the generic small sequence as in the statement. We will trace its journey through the diagram \eqref{diagram}. 
		Since $\bc\times \mathbf{X}$ is not $(\bbphi,\bbpsi)$-preperiodic, 
Proposition \ref{equidistribution} (see also Remark \ref{measures as currents}) yields that the Galois orbits of $z$ equidistribute with respect to 
		$$P:=\frac{1}{v}\widehat{T}^{\wedge (1+\dim \mathcal{X})}_{\Phi\times _{B_0}\Psi}\wedge [\mathcal{C}\times_{B_0}\mathcal{X}],$$
		where $v\in\bR_{>0}$.
Moreover, for each $i=1,2$ the sequence $\tilde{\pi}_i(z):=\{\tilde{\pi}_i(z_n)\}_n$ is $(f_i\times_{B_0}\Psi, \bP_1\times \mathcal{X})$-generic and small. 
Since $\bP_1\times \mathbf{X}$ is not $(\bbf_i,\bbpsi)$-preperiodic, Proposition \ref{equidistribution} yields that its Galois conjugates equidistribute with respect to  $\frac{1}{v_i}\widehat{T}^{\wedge (1+\dim \mathcal{X})}_{f_i\times _{B_0}\Psi}\wedge [\bP_1\times\mathcal{X}]$ for $v_i>0$. 	We infer 
		\begin{align}\label{Pproj}
			\tilde{\pi}_{i*}P=\frac{1}{v_i}\widehat{T}^{\wedge (1+\dim \mathcal{X})}_{f_i\times _{B_0}\Psi}\wedge [\bP_1\times\mathcal{X}],
		\end{align}
		for each $i=1,2$. 
		The aforementioned equidistribution statements further imply that the Galois conjugates of 
		$p\circ \tilde{\pi}_i(z)= \pi_i\circ\pi_{\mathcal{C}}(z)$ equidistribute with respect to 
		\begin{align}\label{same projections}
			p_{*}\tilde{\pi}_{i*}P= \pi_{i*}\pi_{\mathcal{C}*} P.
		\end{align}
Since $\pi_i$ is a finite map, we have $\pi_i^{*} \pi_{i*}\pi_{\mathcal{C}*} P=\deg(\pi_i)\pi_{\mathcal{C}*} P$. 
Recalling \eqref{Pproj} and pulling back \eqref{same projections}  by $\pi_i$ we thus infer 
\begin{align*}
\frac{\deg(\pi_2)}{v_1}\pi_1^{*}p_{*}(\widehat{T}^{\wedge (1+\dim \mathcal{X})}_{f_1\times _{B_0}\Psi}\wedge [\bP_1\times\mathcal{X}]) = \frac{\deg(\pi_1)}{v_2}\pi_2^{*}p_{*}(\widehat{T}^{\wedge (1+\dim \mathcal{X})}_{f_2\times _{B_0}\Psi}\wedge [\bP_1\times\mathcal{X}]).
\end{align*}
The lemma follows.
\end{proof}

Our next goal is to express the measures in Lemma \ref{equidistributiontwice} fiberwise. First, we will reinterpret their definition so that ultimately the characterization of slicing of currents applies.
Our end goal is to infer that \eqref{push and pull} contradicts Proposition \ref{test function} due to our assumption that $\bc$ is not weakly $\bbphi$-special. We first record a couple of lemmata. 

\begin{lemma}\label{blocki}
            With notation as in diagram \eqref{diagram}, there is a constant $\kappa_i>0$ such that $\widehat{T}^{\wedge (1+\dim \mathcal{X})}_{f_i\times _{B_0}\Psi}\wedge [\bP_1\times\mathcal{X}]=\kappa_i ( p^{*}\widehat{T}_{f_i}\wedge q^{*}(\widehat{T}_{\Psi}^{\wedge(\dim\mathcal{X})}\wedge [\mathcal{X}]))$, for  $i=1,2$.
		\end{lemma}
		\begin{proof}
We fix projections $\Pi_1: 	(B_0\times\bP_1)\times_{B_0} (B_0\times\bP_1^{\ell})\to B_0\times \bP_1$ and 
$\Pi_2: (B_0\times\bP_1)\times_{B_0} (B_0\times\bP_1^{\ell})\to B_0\times \bP^{\ell}_1$. 
The map $f_i\times_{B_0}\Phi$ is polarized with respect to $\mathcal{L}=\Pi_1^{*}\mathcal{L}_1\otimes \Pi_2^{*}\mathcal{L}_2$, 
where $\mathcal{L}_1$ is the line bundle associated to the divisor $B_0\times \{\infty\}$ on $B_0\times\bP_1$ and $\mathcal{L}_2$ the line bundle associated to $B_0\times ( \{\infty\}\times\bP_1\cdots\times\bP_1+\cdots+\bP_1\times\cdots\times\mathbb{P}_1\times\{\infty\})$ on $B_0\times \bP_1^{\ell}$. 
Let $\widehat{\omega}_i$ be a continuous positive $(1,1)$-form cohomologous to a multiple of $\mathcal{L}_i$ and normalized as in \eqref{normalize omega}. 
Then $r_1\Pi_1^{*}\widehat{\omega}_1+ r_2\Pi_2^{*}\widehat{\omega}_2$ is cohomologous to $\mathcal{L}$ for constants $r_1,r_2\in\bQ_{>0}$. By the limiting construction of the fiber-wise Green current as in e.g. \cite[Proposition 9]{GV} we infer 
\begin{align}
\widehat{T}_{f_i\times _{B_0}\Psi}=r_1\Pi_1^{*}\widehat{T}_{f_i} + r_2\Pi^{*}_2\widehat{T}_{\Psi},
\end{align}
adjusting the constants $r_1,r_2>0$ if necessary to ensure that both sides are normalized as in \eqref{normalize omega}, 
so that 
\begin{align}\label{split the currents}
\widehat{T}_{f_i\times _{B_0}\Psi}\wedge [\bP_1\times\mathcal{X}]= r_1p^*\widehat{T}_{f_i}+r_2q^{*}(\widehat{T}_{\Psi}\wedge[\mathcal{X}]).
\end{align}
Notice now that since $B_0\times \bP^1$ is $2$-dimensional we have $(p^{*}\widehat{T}_{f_i})^{\wedge k}=p^*(\widehat{T}_{f_i})^{\wedge k}=0$ (in the sense of currents) for $k\ge 3$ and similarly looking at the dimensions we have $q^{*}(\widehat{T}_{\Psi}\wedge[\mathcal{X}])^{\wedge \dim\mathcal{X}+1}= 0$. 
Moreover  
$(p^{*}\widehat{T}_{f_i})^{\wedge 2}= 0$, by the limiting definition of $\widehat{T}_{f_i}$ since $\widehat{\omega}^{\wedge 2}=\pi_{\mathbb{P}_1}^{*}\omega_{FS}^{\wedge 2}=0$, for the projection $\pi_{\mathbb{P}_1}:B_0\times\bP_1\to \bP_1$ and the Fubini-Study form $\omega_{FS}$ on $\bP_1$.
Hence passing to the $(\dim\mathcal{X}+1)$-th power in \eqref{split the currents} we get 
\begin{align}
\widehat{T}^{\wedge (1+\dim \mathcal{X})}_{f_i\times _{B_0}\Psi}\wedge [\bP_1\times\mathcal{X}]=r_1r_2^{\dim\mathcal{X}}(\dim\mathcal{X}+1)( p^{*}\widehat{T}_{f_i}\wedge q^{*}(\widehat{T}_{\Psi}^{\wedge(\dim\mathcal{X})}\wedge [\mathcal{X}])).
\end{align}
The lemma follows. 
\end{proof}

Next we record the following base change formula for currents, which is well-known but we have not found a reference.

\begin{lemma}\label{base change}
Let 
         \begin{equation}
		\begin{tikzcd}
			X'\arrow{d}{f'}  \arrow[r, yshift=0.7ex, "h"]  & X \arrow{d}{f}\\
			Y' \arrow[r, yshift=0.7ex, "g"] & Y
		\end{tikzcd}
	\end{equation}
be a Cartesian diagram of complex varieties, where $f$ and $f^{\prime}$ are proper, $h$ and $g$ are submersions. For every current $T$ on $X$ with compact support, we have 
$$g^*f_*T = f_*^{'}h^*T.$$
\end{lemma}

\begin{proof}
 This follows from the classical base change formula for smooth forms on $Y'$. Namely $ f^{*}g_{*} \omega=h_{*} f^{\prime*} \omega$, where $h_{*} \omega'=\int_{h} \omega'$ denotes the fiber integral. Recall that $h^{*} T(\eta)=T\left(h_{*} \eta\right)$ and $f'_*T(\eta) = T(f^{'*}\eta)$ for a smooth form $\eta$. 
\end{proof}
In what follows we are going to view all projection maps as maps from products of the projective line with $B_0$ to products of the projective line with $B_0$. We may assume that $B_0$ is smooth after possibly taking out some points. This allows us to apply Lemma \ref{base change} to deduce the following. 
\begin{remark}\label{around diagram}
With notation as in diagram \ref{diagram} we have 
	\begin{align*}
		p_{*}q^{*}(\widehat{T}_{\Psi}^{\wedge(\dim\mathcal{X})}\wedge [\mathcal{X}])= p^{*}_Bq_{B*}(\widehat{T}_{\Psi}^{\wedge(\dim\mathcal{X})}\wedge [\mathcal{X}]). 
		\end{align*}
\end{remark}
 \noindent Finally, for the convenience of the reader, we recall  here the following typical  approximation lemma that is already implicit in the construction of the dynamical currents used above. A reference is \cite[Theorem 3.3.3]{DS:course}. 
\begin{remark}\label{approxT}
    There is a sequence of smooth currents $T_n$ with compact support such that $T_n$ converges  to $\hat{T}^{\wedge \dim(\mathcal{X})}_{\Psi}\wedge[\mathcal{X}]$. 
\end{remark}

  \noindent  We are now ready to prove Theorem \ref{relative bogomolov}. 
    
   \smallskip 
    
 \noindent\emph{Proof of Theorem \ref{relative bogomolov}:}
 Under the assumptions in Theorem \ref{relative bogomolov}, assume that there is a $(\Phi\times_{B_0}\Psi,\mathcal{C}\times_{B_0}\mathcal{X})$-generic small sequence to end in a contradiction. 
 Combining Lemmata \ref{equidistributiontwice} and \ref{blocki} and the observation in Remark \ref{around diagram} and using the projection formula for currents we infer  that 
 \begin{align}\label{alp}
 \pi_1^*\widehat{T}_{f_1}\wedge \pi_{1}^*p^{*}_B(q_{B*}(\widehat{T}_{\Psi}^{\wedge(\dim\mathcal{X})}\wedge [\mathcal{X}])) = \beta \pi_2^*\widehat{T}_{f_2}\wedge \pi_{2}^*p^{*}_B(q_{B*}(\widehat{T}_{\Psi}^{\wedge(\dim\mathcal{X})}\wedge [\mathcal{X}])),
 \end{align}
 for $\beta\in \bR_{>0}$. 
 To simplify our notation we let $\Pi_{\mathcal{C}}:=p_B\circ\pi_i:\mathcal{C}\to B_0$ and write 
 \begin{align}\label{R} 
  R := q_{B*}(\widehat{T}_{\Psi}^{\wedge(\dim\mathcal{X})}\wedge [\mathcal{X}]).
  \end{align} 
In this notation equation \eqref{alp} reads as
\begin{align}\label{alpha}
   \pi_1^*\widehat{T}_{f_1}\wedge\Pi^{*}_{\mathcal{C}}R=\beta   \pi_2^*\widehat{T}_{f_2}\wedge\Pi^{*}_{\mathcal{C}}R.
\end{align}

 Let $E_{\bbphi, \bc}$ be the finite set as in Definition \ref{finite set}. 
 	Since $\mathbf{X}$ is not $\bbpsi$-preperiodic, by Proposition \ref{compactum} we can find a compact set $\mathcal{D}\subset B_0\setminus E_{\bbphi, \bc}$ such that 
 	\begin{align}\label{positive base}
 		\int_{\mathcal{D}}R>0. 
 	\end{align}
 We will first show that $\beta=\deg(\pi_1)/\deg(\pi_2)$. 
 Let $\chi_{\mathcal{D}}:B_0\to \bR$ be a non-negative smooth function which is 1 on $\mathcal{D}$ and 0 outside some compact set $\mathcal{D}^{\prime}$ with $\mathcal{D} \subseteq \mathcal{D}^{\prime} \subseteq B_{0} \backslash E_{\boldsymbol{\Phi}, \mathbf{C}}$. Then \eqref{alpha} yields 
 \begin{align}\label{find alpha}
 \int_{\mathcal{C}}\chi_{\mathcal{D}}\circ\Pi_{\mathcal{C}}\wedge \pi_1^*\widehat{T}_{f_1}\wedge \Pi_{\mathcal{C}}^*R =\beta \int_{\mathcal{C}}\chi_{\mathcal{D}}\circ\Pi_{\mathcal{C}}\wedge \pi_2^*\widehat{T}_{f_2}\wedge \Pi_{\mathcal{C}}^*R.
 \end{align}
 Let us pretend for a moment that $R$ is smooth. 
 Then by the characterization of slicing as in \cite[Proposition 4.3]{BB} equation \eqref{find alpha} would yield
 \begin{align}\label{slice pretend}
 \int_{B_0}\left(\chi_{\mathcal{D}}(t)\int_{\mathcal{C}_t} \pi_{1,t}^*T_{f_{1,t}}\right)R =\beta \int_{B_0}\left(\chi_{\mathcal{D}}(t)\int_{\mathcal{C}_t} \pi_{2,t}^*T_{f_{2,t}}\right)R,
 \end{align}
 where $T_{f_{i,t}}$ denotes the slice of $\widehat{T}_{f_{i}}$ above $t$. 
 We recall that this slice is well-defined and furthermore $T_{f_{i,t}}=\mu_{f_{i,t}}$; see e.g. \cite[Propositions 4 and 9]{GV}. 
 By Wirtinger's formula (see e.g. \cite{ChambertLoir:survey}) we further have $\int _{\mathcal{C}_t} \pi_{1,t}^*T_{f_{1,t}}=\deg(\pi_{i,t})=\deg\pi_i$ for $t\in \mathcal{D}$, so that recalling \eqref{positive base} equation \eqref{slice pretend} yields that $\beta=\deg(\pi_1)/\deg(\pi_2)$ as claimed.  
 
 Of course, we have been falsely assuming that $R$ is smooth. To reconcile this we will  reduce to the smooth case. We let $\{R_n\}_{n\in\bN}$ be a sequence of smooth currents such that 
 \begin{align*}
 R_n \rightarrow q_{B*}(\widehat{T}_{\Psi}^{\wedge(\dim\mathcal{X})}\wedge [\mathcal{X}]) \text{;  }
 \pi_i^{*}\widehat{T}_{f_i}\wedge \pi_{i}^*p^*_BR_n& \rightarrow  \pi_1^{*}\widehat{T}_{f_i}\wedge \pi_{i}^*p^*_Bq_{B*}(\widehat{T}_{\Psi}^{\wedge(\dim\mathcal{X})}\wedge [\mathcal{X}]),
 \end{align*}
 where $R_n$ can be given by replacing $\widehat{T}_{\Psi}^{\wedge(\dim\mathcal{X})}\wedge [\mathcal{X}]$ with $T_n$ as in Remark \ref{approxT}.
 From the definition of slicing \cite[Proposition 4.3]{BB} we have that 
 $$\int_{\mathcal{C}}u\pi_i^*\widehat{T}_{f_i}\wedge \pi_{\mathcal{C}}^*R_n = \int_{B_0}\left(\int_{\mathcal{C}_t}u_{|\mathcal{C}_t}\pi_{i,t}^*T_{f_{i,t}}\right)R_n,$$
 for every smooth compactly supported function $u$ on $\mathcal{C}$ and all $n\in\bN$. 
 As the sequence of the currents $R_n$ converges to $R$ we infer that the same equality holds with $R_n$ replaced by $R$. The operator 
 $$\mathfrak{O}(u) = \int_{B_0}\left(\int_{\mathcal{C}_t}u_{|\mathcal{C}_t}\pi_{i,t}^*T_{f_{i,t}}\right)R$$
 is defined on smooth compactly supported functions and we established that its value equals $\int_{\mathcal{C}}u\pi_i^*\widehat{T}_{f_i}\wedge \pi_{\mathcal{C}}^*R$. Following the proof of  \cite[Proof of Proposition 4.3, Appendix]{BB} we see that $\mathfrak{O}$ can be continuously extended to continuous compactly supported functions. 
 From the continuity of the current $\pi_i^*\widehat{T}_{f_i}\wedge \pi_{\mathcal{C}}^*R$ it follows that 
 \begin{align}\label{cont} \int_{\mathcal{C}} u\pi_i^*\widehat{T}_{f_i}\wedge \pi_{\mathcal{C}}^*R = \int_{B_0}\left(\int_{\mathcal{C}_t}u_{|\mathcal{C}_t}\pi_{i,t}^*T_{f_{i,t}}\right)R,
 	\end{align}
 for all continuous and compactly supported functions $u$ and $i=1,2$.  Setting $u = \chi_{\mathcal{D}}$ we arrive at the equality \eqref{slice pretend} and thus we have \eqref{alpha} with $\beta= \deg(\pi_1)/\deg(\pi_2)$. 
 
 Similarly, using \eqref{cont} with $u$ being the continuous function $\psi_{\mathcal{D}}$ corresponding to $\mathcal{D}$ from Proposition \ref{test function} we infer from \eqref{alpha} that
 \begin{align}\label{totrans}
 \int_{\mathcal{D}}\left(\int_{\mathcal{C}_t} \psi_{\mathcal{D}}|_{\mathcal{C}_t} \deg(\pi_{2})d\pi_{1,t}^*\mu_{f_{1,t}}(z)\right)R = \int_{\mathcal{D}}\left(\int_{\mathcal{C}_t} \psi_{\mathcal{D}}|_{\mathcal{C}_t} \deg(\pi_{1})d\pi_{2,t}^*\mu_{f_{2,t}}(z)\right)R,
 \end{align}
 Recalling that $\mathcal{D}$ is chosen as in \eqref{positive base} and $\psi_{\mathcal{D}}$ satisfies the statement of Proposition \ref{test function}, equation \eqref{totrans} implies that the set $\mathcal{P}:=\{t\in \mathcal{D}: \mathcal{C}_t \text{ is weakly }\Phi_t\text{-special}\}$ has positive $R$ measure. 
 In particular since $R$ doesn't charge points (e.g. by \cite[Proposition 4.6.4]{KL}), the set $\mathcal{P}$ is uncountable and must therefore contain a transcendental parameter $t_0\in B_0(\bC)\setminus B_0(\Qbar)$.  But since $B$ is $1$-dimensional this contradicts our assumption that $\bc$ is not weakly $\bbphi$-special. 
 The theorem follows. 	\qed 
  
\smallskip  
  
\noindent In fact, we also infer the following result. 

\begin{theorem}\label{relative bogomolov iso}
	Let $\bbphi=(\bbf_1,\bbf_2): \mathbb{P}_1^2 \rightarrow \mathbb{P}_1^2$ be a split polarized endomorphism of polarization degree $d\ge 2$ and assume that $\mathbf{f}_i$ is isotrivial for some $i\in\{1,2\}$. 
    Let $\mathbf{C} \subset \mathbb{P}_1^2$ be an irreducible curve that is not weakly $\bbphi$-special. 
    Then there exists $\epsilon=\epsilon(\bbphi,\bc)> 0$ such that 
$$	\{(P_1,P_2) \in \mathcal{C}\times_{B_0}\mathcal{C}(\overline{\Q})~:~\hat{h}_{\Phi}(P_1) +  \hat{h}_{\Phi}(P_2)  < \epsilon  \}$$ 
is not Zariski dense in $\mathcal{C}\times_{B_0}\mathcal{C}$. 
\end{theorem}

\begin{proof} 

If $\bc$ is not $\bbphi$-isotrivial then the result follows exactly as Theorem \ref{relative bogomolov}. 
If $\bc$ is $\bbphi$-isotrivial then the statement follows from \cite[Theorem 1.1]{dmm1}, \cite[Theorem 6.2]{Zhang:positivesurfaces} and \cite[Corollary 3]{ZhangBog}.
\end{proof} 

\smallskip

\noindent Finally, we deduce Theorem \ref{umb} and its Corollary \ref{small curves}. 

\smallskip

\noindent\emph{Proof of Theorem \ref{umb}: }
Let $\bbphi$ and $\bc$ be as in the statement of Theorem \ref{umb}. 
Assume that Theorem \ref{relative bogomolov} (or Theorem \ref{relative bogomolov iso} in the non-isotrivial-free case) holds for $\bbphi=\bbpsi:\bP_1^2\to \bP_1^2$, for $\mathbf{X}=\bc$ and with constant $\epsilon>0$. So we know that the Zariski closure $\mathcal{M}$ of the set
$	\{(P,Q)\in(\mathcal{C}\times_{B_0}\mathcal{C})(\Qbar)~:~\hat{h}_{\Phi}(P)+\hat{h}_{\Phi}(Q)<\epsilon\} $
	has dimension at most equal to $2$. 
	For each $t\in B_0(\Qbar)$ we let 
	\begin{align*}
		\Sigma_t:=\{z\in \mathcal{C}_t(\Qbar)~:~\hat{h}_{\Phi_t}(z)<\epsilon/2 \}.
	\end{align*}	
    Clearly we have that 
	\begin{align*}
		\Sigma_t\times \Sigma_t\subset \mathcal{M}_t\subset \mathcal{C}_t\times \mathcal{C}_t.
	\end{align*}
Let $\pi_1:\mathcal{M}\to \mathcal{C}$ denote the projection to the first factor. We write $\mathcal{M}=\mathcal{M}_1\cup \mathcal{M}_2$, where $\mathcal{M}_1$ consists of the irreducible components $\mathcal{I}\subset \mathcal{M}$ with $\pi_1(\mathcal{I})=\mathcal{C}$ and $\mathcal{M}_2$ consists of the remaining components. Let $T\subset B_0(\Qbar)$ be the finite set obtained by the image of the vertical irreducible components of $\mathcal{M}$. Then for every $t\in B_0(\Qbar)\setminus T$ the size of the set $\pi_{1,t}(\mathcal{M}_{2,t})$ is uniformly bounded by some constant $D(\mathcal{M}_2)$. We write 
$$\Sigma_{1,t}:=\Sigma_t\setminus(\Sigma_t\cap \pi_{1,t}(\mathcal{M}_{2,t})).$$
As above, we have 
$$\Sigma_{1,t}\times \Sigma_{1,t}\subset \mathcal{M}_{1,t}\subset \mathcal{C}_t\times \mathcal{C}_t.$$
Since all components in $\mathcal{M}_1$ surject onto the first factor of $\mathcal{C}\times_{B_0}\mathcal{C}$, there is a uniform constant $D(\mathcal{M}_1)$ such that 
$$\# \Sigma_{1,t}=\#(\{P\}\times\Sigma_{1,t})\le \#(\pi^{-1}_{1,t}(\{P\})\cap \mathcal{M}_{1,t})\le D(\mathcal{M}_1)$$
for all $t\in B_0(\Qbar)$ and all $P\in\Sigma_{1,t}$. Finally, we conclude that 
$$\#\Sigma_t\le\# \Sigma_{1,t}+D(\mathcal{M}_2)\le D(\mathcal{M}_1)+D(\mathcal{M}_2)$$
for all $t\in B_0(\Qbar)\setminus T$. The theorem follows.
\qed

\smallskip

\noindent\emph{Proof of Corollary \ref{small curves}:}
Recall the definition of the essential minima from \eqref{minima}. 
Let $t\in B_0(\Qbar)$ and write $\overline{L}_{\Phi_t}$ for the invariant adelic metrized line bundle extending $L={\infty}\times\bP_1+ \bP_1\times\{\infty\}$. 
Clearly $e_2(\mathcal{C}_t, \overline{L}_{\Phi_t})\ge 0$. 
Moreover, by Theorem \ref{umb} there is $\epsilon>0$ such that 
 $e_1(\mathcal{C}_t, \overline{L}_{\Phi_t})\ge \epsilon$ for all but finitely many $t\in B_0(\Qbar)$. 
 By Zhang's inequality in Lemma \ref{zhang ineq} we deduce that
 $\hat{h}_{ \overline{L}_{\Phi_t}}(\mathcal{C}_t)=\hat{h}_{\Phi_t}(\mathcal{C}_t)\ge \frac{\epsilon}{2}$ for all but finitely many $t$. This proves $(2)$.

For $(1)$, note that if $t\in B_0(\Qbar)$ is such that $\mathcal{C}_t$ is $\Phi_t$-preperiodic, then by Zhang's inequality and \cite[Lemma 6.1]{dmm1} we have $\hat{h}_{\Phi_t}(\mathcal{C}_t)=0$. 
On the other hand if there is $t\in B_0(\bC)\setminus B_0(\Qbar)$ such that $\mathcal{C}_t$ is $\Phi_t$-preperiodic then since $B$ is $1$-dimensional, $\bc$ would be $\bbphi$-preperiodic, contradicting our assumption that is not weakly $\bbphi$-special. 
This completes the proof of the corollary. 
\qed

\section{Common preperiodic points}\label{applications}
\noindent In this section we establish Theorem \ref{moduli:curve}. 
   
 \smallskip 
   
\noindent\emph{Proof of Theorem \ref{moduli:curve}:}
Let $C\subset \mathrm{Rat}_d\times \mathrm{Rat}_d$ be a curve defined over $\Qbar$ and let $p_i: C\to  \mathrm{Rat}_d$ be its projection to the $i$-th copy of $\text{Rat}_d$. For $\underline{\mu}\in C(\mathbb{C})$ let $f_{p_i(\underline{\mu})}$ denote the rational map associated to $p_i(\underline{\mu})\in \mathrm{Rat}_d(\mathbb{C})$ for $i=1,2$.
We will show that there exists $M = M(C)>0$ such that 
              	$$\text{either }\#\mathrm{Prep}(f_{p_1(\underline{\mu})})\cap \mathrm{Prep}(f_{p_2(\underline{\mu})})<M\text{ or }\mathrm{Prep}(f_{p_1(\underline{\mu})})=\mathrm{Prep}(f_{p_2(\underline{\mu})}).$$
In the setting of Theorem \ref{umb} we set $B = C$ and recall that $K = \overline{\Q}(B)$. 
Let $\bbphi = (f_{p_1(\mathbf{c})} ,f_{p_2(\mathbf{c})})$ where $\mathbf{c}$ is the generic point of $C$. 

We first prove that  the statement holds if we restrict to transcendental points $\underline{\mu} \in C(\C)\setminus C(\Qbar)$. 
In fact, then $\underline{\mu}$ is the embedded image of the generic point of $C$ into $\mathbb{C}$ (so it induces an embedding $K \hookrightarrow \mathbb{C}$) and Theorem \ref{moduli:curve} reduces to the dynamical Manin-Mumford conjecture over $\mathbb{C}$ \cite{dmm1} if one of the maps $(f_{p_1(\mathbf{c})} ,f_{p_2(\mathbf{c})})$ is ordinary or if both  are exceptional to the classic Manin-Mumford conjecture \cite{Hindry}. Note that if $\mathbf{\Delta}$ is $\bbphi$-special then $\mathrm{Prep}(f_{p_1(\underline{\mu})}) = \mathrm{Prep}(f_{p_2(\underline{\mu})})$ for all $\underline{\mu}\in C(\mathbb{C})$.

We may thus restrict our attention to algebraic points in $C(\Qbar)$ and we will consider two cases depending on whether $\bbphi$ is isotrivial or not. 
Assume first that  $\bbphi$ is not isotrivial. 
Then by Theorem \ref{umb} we infer that there exists a bound $M$ for all algebraic points $C(\overline{\Q})$ as in Theorem \ref{moduli:curve} outside a finite set of exceptions $E \subset C(\overline{\Q})$ unless $\mathbf{\Delta}$ is weakly $\bbphi$-special. 
As $\bbphi$ is not isotrivial, $\mathbf{\Delta}$ is weakly $\bbphi$-special if and only if it is $\bbphi$-preperiodic. 
If for some $\underline{\mu} \in E$, the diagonal $\mathbf{\Delta}$  is not $(f_{p_1(\underline{\mu})}, f_{p_2(\underline{\mu})})$-special we can apply  the main theorem of \cite{dmm1} or \cite{Hindry} and adjust the bound $M$ accordingly.  If on the other hand, $\mathbf{\Delta}$ is $(f_{p_1(\underline{\mu})}, f_{p_2(\underline{\mu})})$-special  then $\mathrm{Prep}(f_{p_1(\underline{\mu})}) = \mathrm{Prep}(f_{p_2(\underline{\mu})}) $. Thus Theorem \ref{moduli:curve} is proven if $\bbphi$ is not isotrivial. 

Now assume that $\bbphi$ is isotrivial. 
Then there exists a pair of Möbius transformation $(\mathbf{M}_1,\mathbf{M}_2)$  such that $\mathbf{M}_i\circ f_{p_i(\mathbf{c})}\circ \mathbf{M}_i^{-1}$ is defined over $\overline{\Q}$ for $i=1,2$. 
As at least one $p_1,p_2$ is non-constant, at least one of $\mathbf{M}_1, \mathbf{M}_2$ is not defined over $\overline{\Q}$. Thus $(\mathbf{M}_1, \mathbf{M}_2)(\mathbf{\Delta})$ is not defined over $\overline{\Q}$ and $\mathbf{\Delta}$ is not $\bbphi$-isotrivial. 
If $\mathbf{\Delta}$ is not weakly $\bbphi$-special, then the result follows by Theorem \ref{umb}. 
If on the other hand $\mathbf{\Delta} $ is weakly $\bbphi$-special, then it is the image of a coset in a fixed algebraic group as in Definition \ref{weakly special}. 
If for $\underline{\mu}\in C(\overline{\mathbb{Q}})$ the corresponding coset is a torsion coset (in its fiber) then $\mathrm{Prep}(f_{p_1(\underline{\mu})}) = \mathrm{Prep}(f_{p_2(\underline{\mu})})$. If it is not a torsion coset then $|\mathrm{Prep}(f_{p_1(\underline{\mu})}) \cap \mathrm{Prep}(f_{p_2(\underline{\mu})})| \leq 4$ as $\Delta\cap \mathrm{Prep}(f_{p_1(\underline{\mu})}) \times \mathrm{Prep}(f_{p_2(\underline{\mu})}) \subset \mathbb{P}_1^2 \setminus (\Psi_{f_{p_1(\underline{\mu})}}(G_{f_{p_1(\underline{\mu})}})\times  \Psi_{f_{p_2(\underline{\mu})}}(G_{f_{p_2(\underline{\mu})}})) $. This concludes the proof. 
 \qed

\smallskip 


\renewcommand\thesection{\Alph{section}}
\setcounter{section}{1}
\section*{ Appendix: A specialization result}\label{general inequality}
	
	In this appendix we establish a dynamical analogue of \cite[Theorem 1.4]{Gao:Habegger:bog} for $1$-parameter families of polarized endomorphisms (not necessarily split). 
	More precisely, in Theorem \ref{assume bog} we provide a relation between the fiber-wise Call-Silverman canonical height and a height on the base, generalizing Call-Silverman's result \cite{Call:Silverman} to higher dimensions. 
    Yuan and Zhang \cite[Theorem 5.3.5 and Theorem 6.2.2]{Yuan:Zhang:new} have recently obtained a more general result using their new theory of adelic metrized line bundles on quasi-projective varieties. 
    Here we rely on simpler machinery. Our approach is adapted from the strategy 
	in \cite{DGH:pencils, DGH:uml}, that in turn built on Habegger's work \cite{Habegger:special}. 
	
	To fix notation, let $B$ be a projective, regular, irreducible curve over $\Qbar$ and write $K=\Qbar(B)$.
	Let $\mathbf{A}$ be a projective normal variety and $L$ be a very ample line bundle on $\mathbf{A}$, both defined over $K$. 
	Let $\bbphi: \mathbf{A}\to \mathbf{A}$  be an endomorphism defined over $K$ such that $\bbphi^{*}(L)\iso L^{\otimes d}$ for $d\in \bZ_{\ge 2}$. 
	We call the triple $(\mathbf{A},\bbphi,L)$ a \emph{polarized endomorphism} over $K$ with polarization degree $d$. 
	
	We note that $\bbphi$ can be viewed as a $1$-parameter family of endomorphisms each defined over $\Qbar$. 
	To this end, consider a $B_0$-model $(\mathcal{A},\Phi,\mathfrak{L})$ of $(\mathbf{A},\bbphi,L)$, where $B_0\subset B$ is Zariski open.
	More specifically, $\mathcal{A}$ is a variety defined over $\Qbar$, $\pi:\mathcal{A}\to B_0$ is normal and flat and has generic fiber $\mathbf{A}$, $\Phi: \mathcal{A}\to \mathcal{A}$ is the endomorphism associated to $\bbphi$, the generic fiber of $\mathfrak{L}$ is $L$ and for each $t\in B_0$ the restriction of $\mathfrak{L}$ on the fiber of $\mathcal{A}$ over $t$ is a very ample line bundle 
	such that $(\Phi|_{\pi^{-1}(t)})^{*}(\mathfrak{L}|_{\pi^{-1}(t)})\iso (\mathfrak{L}|_{\pi^{-1}(t)})^{\otimes d}$. 
	For simplicity we write $\Phi_t:=\Phi|_{\pi^{-1}(t)}$, $\mathcal{A}_t:=\mathcal{A}|_{\pi^{-1}(t)}$ and $\mathfrak{L}_t:=\mathfrak{L}|_{\pi^{-1}(t)}$. 
	Thus for each $t\in B_0(\Qbar)$ the triple $(\mathcal{A}_t,\Phi_t,\mathfrak{L}_t)$ is a polarized endomorphism of polarization degree $d$ defined over a number field. In particular we have a fiberwise Call-Silverman canonical height defined for algebraic points $P\in \mathcal{A}^0=\pi^{-1}(B_0)$ by 
	$$\hat{h}_{\Phi,\mathfrak{L}}(P)=\hat{h}_{\Phi_{\pi(P)},\mathfrak{L}_{\pi(P)}}(P),$$
	where $\hat{h}_{\Phi_{t},\mathfrak{L}_{t}}$ is the canonical height defined in \cite{Call:Silverman} associated to the endomorphism $\Phi_{t}:\mathcal{A}_t\to \mathcal{A}_t$ for $t\in B_0(\Qbar)$ with a polarization given by $\mathfrak{L}_t$.
	
	Recall that for an irreducible subvariety $\mathbf{X}$ of $\mathbf{A}$ defined over $K$ we write $h_{\overline{L}_\bbphi}(\mathbf{X})\ge 0$ for the height of $\mathbf{X}$ with respect to the canonical metric $\overline{L}_{\bbphi}:=(L,\{\|\cdot\|_{\bbphi,L,v}\}_v)$ as in \S\ref{background}. 
    We write $\mathcal{X}\subset \mathcal{A}$ for the Zariski closure of $\mathbf{X}$ in $\mathcal{A}$.
	We can now state our theorem. 
	
	\begin{theorem}\label{assume bog}
		Let $(\mathcal{A},\Phi,\mathfrak{L})$ be a $B_0$-model of a polarized endomorphism $(\mathbf{A},\bbphi,L)$ over $K$ with polarization degree $d\ge 2$ and let $\mathbf{X}\subset \mathbf{A}$ be an irreducible subvariety. 
        Let also $\mathcal{N}$ be an ample line bundle on $B$. Assume that $h_{\overline{L}_{\bbphi}}(\mathbf{X})>0$. Then there are constants $\chi:=\chi(\mathbf{X},\bbphi,L)>0$, and $m:=m(\mathbf{X},\bbphi,L,\mathcal{N},\epsilon)>0$ and a Zariski open subset $\mathcal{U}:=\mathcal{U}(\bbphi, \mathbf{X},\mathcal{N})\subset \mathcal{X}$ such that 
		\begin{align*}
			\hat{h}_{\Phi,\mathfrak{L}}(P)\ge \chi h_{\mathcal{N}}(\pi(P))-m,
		\end{align*}
		for all $P\in \mathcal{U}(\Qbar)$. 
	\end{theorem}
	
 \begin{remark}
    The constant $\chi$ in Theorem \ref{assume bog} can be taken as 
    $$\chi= \frac{\deg_{L}(\mathbf{X})\cdot h_{\overline{L}_{\bbphi}}(\mathbf{X})}{2c},$$ where $c:=c(\bbphi,L,\deg_{L}(\mathbf{X}))>0$ is described in \eqref{effective}.
    \end{remark}

   The proof of this theorem relies on the functoriality properties of dynamical canonical heights on projective varieties and on Siu's bigness inequality \cite[Theorem 2.2.15]{Lazarsfeld:Positivity:I}. 
	To employ the height machine, we will need to work on projective varieties. 
	To this end, following \cite{Xander}, we define inductively a sequence $(\mathcal{A}_N,\mathcal{L}_N)$ of $B$-models of $(\mathbf{A},L^{\otimes d^{N-1}})$ such that $\mathcal{L}_N$ is nef  for each $N$ as follows: 
	For $N=1$ note that $\mathbf{A}$ is projective and $L$ is very ample so we may embed $\mathbf{A}$ into a projective space over $K$ by $L$ and subsequently identify $\mathbb{P}^n_{K}$ with the generic fiber of $\mathbb{P}^n\times B$ to get 
	$$\mathbf{A}\hookrightarrow \mathbb{P}^n_{K}\hookrightarrow \mathbb{P}^n\times B.$$
	We let $\mathcal{A}_1$ be the Zariski closure of $\mathbf{A}$ in $\mathbb{P}^n\times B$
	and write $\mathcal{L}=\mathcal{O}_{\mathcal{A}_1}(1)$. 
	Letting $\pi: \mathcal{A}_1\to B$ denote the projection to the second factor, we may choose a very ample line bundle $\mathcal{M}$ on $B$ such that $\mathcal{L}\otimes \pi^{*}\mathcal{M}$ is nef; see the proof of \cite[Lemma 3.1]{Xander}. We let $\mathcal{L}_1=\mathcal{L}\otimes \pi^{*}\mathcal{M}$ and have a $B$-model of $(\mathbf{A},L)$ given by $(\mathcal{A}_1,\mathcal{L}_1)$.

	For the inductive step suppose that $(\mathcal{A}_N,\mathcal{L}_N)$ is a $B$-model of $(\mathbf{A},L^{\otimes d^{N-1}})$ such that $\mathcal{L}_N$ is nef and let $j_N:\mathbf{A} \hookrightarrow \mathcal{A}_N$ be the inclusion of the generic fiber. Let 
	$$\Gamma_{\bbphi}=(i_d, \bbphi): \mathbf{A}\to \mathbf{A}\times_K \mathbf{A}$$
	be the graph morphism on the generic fiber and write 
	$$\tilde{\Gamma}_{\bbphi,N}= (j_N,j_N)\circ \Gamma_{\bbphi}: \mathbf{A}\to \mathcal{A}_{N}\times_B\mathcal{A}_N.$$
	Then define $\mathcal{A}_{N+1}$ as the Zariski closure of $\tilde{\Gamma}_{\bbphi,N}(\mathbf{A})$ in $\mathcal{A}_{N}\times_B\mathcal{A}_N$ and let $\mathcal{L}_{N+1}=(\mathrm{pr}^{*}_2\mathcal{L}_N)|_{\mathcal{A}_{N+1}}$, where $\mathrm{pr}_2: \mathcal{A}_{N}\times_B\mathcal{A}_N\to\mathcal{A}_N$ is the projection to the second factor. 
	Note that $\mathcal{L}_{N+1}$ is nef as the pullback of a nef line bundle. 
	If $\mathbf{X}$ is a subvariety of $\mathbf{A}$ over $K$, we denote by $\mathcal{X}_N$ its Zariski closure in $\mathcal{A}_N$. 
    We can now describe the $h_{\overline{L}_{\bbphi}}$-height of our subvariety as a limit of intersection numbers in these models. 
	
	\begin{lemma}\label{height positive}
		Let $(\mathbf{A},\bbphi,L)$ be a polarized endomorphism over $K$ of polarization degree $d\ge 2$.
		Let $\mathbf{X}$ be an irreducible subvariety of $\mathbf{A}$ of dimension $d_X\ge 0$. 
		Then we have 
		\begin{align*}
			h_{\overline{L}_{\bbphi}}(\mathbf{X})=\frac{1}{\deg_{L}(\mathbf{X})}\lim_{N\to\infty}\frac{(\mathcal{L}^{d_X+1}_N\cdot \mathcal{X}_N)}{(d_X+1)d^{(N-1)(d_X+1)}}.
		\end{align*}
	\end{lemma}
	
	\begin{proof}
		The lemma follows from the work in \cite{Xander, Gubler:equi}. 
		For each $N\ge 0$ we have constructed $B$-models $(\mathcal{A}_N,\mathcal{L}_N)$ of $(A, L^{\otimes d^{N-1}})$. These induce semipositive continuous adelic metrics $\{\|\cdot\|_{\mathcal{L}_N,v}\}_{v}$ on $L^{\otimes d^{N-1}}$ as in \cite[page 351]{Xander}, which we denote by $\overline{L}_N$. 
		As shown  in \cite[\S 3.1]{Xander} the sequence of metrized line bundles $\{\overline{L}^{\otimes 1/d^{N-1}}_N\}_N$ converges uniformly to the $\bbphi$-invariant metric $\overline{L}_{\bbphi}$ supported on $L$ and moreover outside of finitely many places $v$ of $K$ we have $\|\cdot\|^{1/d^{N-1}}_{\overline{L}_N,v}=\|\cdot\|_{\overline{L}_{\bbphi},v}$ for all $N$. The lemma then follows  in view of \cite[Theorem 3.5 (c)]{Gubler:equi} or \cite[Theorem 2.1 (ii)]{Xander} (see also the discussion around equation (3) in \cite{Xander}). 
	\end{proof}
	
	
	We proceed to clarify the construction of our models and describe the heights associated to them.  
	Let $\mathbf{\iota}: \mathbf{A}\hookrightarrow \mathbb{P}^n_K$ be the embedding over $K$ induced by $L$. 
	We first note that by \cite[Corollary 2.2]{Fakhruddin} we can find an endomorphism $\mathbf{F}:\mathbb{P}^n\to \mathbb{P}^n$ defined over $K$ and such that $\mathbf{F}\circ\iota= \iota\circ \bbphi$.
	Spreading it out and shrinking $B_0$ if necessary we have for each $t\in B_0$ an embedding $\iota_t: \mathcal{A}_t\hookrightarrow \mathbb{P}^n$ given by $\mathfrak{L}_t$ and an endomorphism $F_t:\mathbb{P}^n\to \mathbb{P}^n$ such that 
	$F_t\circ\iota_t= \iota_t\circ \Phi_t.$
	Note here that $\mathfrak{L}= \iota^{*}(\mathcal{L})$. 
	We write $\mathcal{A}^0=\pi^{-1}(B_0)$ to be the part of $\mathcal{A}$ on which $\Phi$ is an endomorphism.  
	For $k\in\mathbb{N}$ we let $(\mathbb{P}^n\times B_0)^{\times k}$ denote the $k$th fibered power over $B_0$ and write 
	$F^{\times k}: (\mathbb{P}^n\times B_0)^{\times k}\to (\mathbb{P}^n\times B_0)^{\times k}$ to be the endomorphism acting as $F_t$ on each fiber $\mathbb{P}^n$ while fixing $t\in B_0$.
	By our definition 
	$\mathcal{A}_N\subset (\mathbb{P}^n\times B)^{\times 2^{N-1}}$ is the Zariski closure of $\mathcal{A}^0_N$, where $\mathcal{A}^0_N$ is defined inductively as follows: 
	\begin{align}
		\begin{split}
			\mathcal{A}^0_1&=\{(\iota_t(z),t): (z,t)\in \mathcal{A}^0\},\\
			\mathcal{A}^0_{N+1}&=\{(Q,F^{\times 2^{N-1}}(Q)): Q\in \mathcal{A}^0_{N}\}.
		\end{split}
	\end{align}
	Recalling that $\mathcal{L}_{N+1}=(\mathrm{pr}^{*}_2\mathcal{L}_N)|_{\mathcal{A}_{N+1}}$, the functoriality of heights  allows us to choose the heights associated to $\mathcal{L}_{N+1}$ inductively so that 
	$$h_{\mathcal{L}_{N+1}}(Q,F^{\times 2^{N-1}}(Q))= h_{\mathcal{L}_{N}}(F^{\times 2^{N-1}}(Q))+O_N(1),$$
	for $Q\in \mathcal{A}^0_{N}(\Qbar)$ with a bounded error term depending on $N$ but independent of $Q$. 
	More specifically if for $Z:=(z,t)\in \mathcal{A}_t(\Qbar)\times B_0(\Qbar)$ we define $Z_{N+1}\in \mathcal{A}^0_{N+1}(\Qbar)$ as the (unique) point in $\mathcal{A}^0_{N+1}(\Qbar)$ determined by $Z$ (its `first' projective coordinate is $\iota_t(z)$ and `last' projective coordinate is given by $F^{N}_t(\iota_t(z))$), then we have 
	\begin{align}\label{height}
		\begin{split}
			h_{\mathcal{L}_{N+1}}(Z_{N+1})&=h_{\iota^{*}\mathcal{L}_1}(\Phi^{N}(Z))+O_N(1)= h_{\iota^{*}_t\mathcal{L}_t}
			(\Phi^N_t(z))+ h_{\mathcal{M}}(t) +O_N(1)\\
			&= h_{\mathfrak{L}_t}(\Phi^N_t(z))+ h_{\mathcal{M}}(t)+O_N(1).
		\end{split}
	\end{align}
	It also follows from our construction that $\mathcal{X}_N\subset \mathcal{A}_N$ defined earlier is the Zariski closure of $\mathcal{X}^0_N$, where for each $N$ the set $\mathcal{X}^0_{N+1}$ is the graph of $\mathcal{X}^0_N$ under $F^{\times 2^{N-1}}$.

	Finally, we write $f_N: \mathcal{A}_N\to B$ and define $\mathcal{M}_N=(f_N)^{*}(\mathcal{M})$. 
	We will need the following bound on intersection numbers. 
	
	\begin{lemma}\label{upper bound}
		We have $(\mathcal{M}_N\cdot \mathcal{L}^{\dim \mathbf{X}}_N\cdot \mathcal{X}_N) \le c d^{(N-1)\dim \mathbf{X}}$ for a constant $c:=c(\bbphi,L,\mathbf{X})>0$.
	\end{lemma}
	Our proof is adapted from \cite[Proposition 4.3]{DGH:uml}, which was inspired by K\"uhne's work on semiabelian varieties \cite{Kuhne:semiabelian}. A slightly different approach to this upper bound was employed in Habegger's earlier works \cite{Habegger:boundedheight, Habegger:special}. An approach relying on currents is demonstrated in \cite{GV}. 
    Using the work in \cite{GV}, Gauthier \cite{Gauthier:equi} has obtained \cite[Lemma 17]{Gauthier:equi} a related height inequality in the case of points $\mathbf{X}$.
	\begin{proof}[Proof of Lemma \ref{upper bound}]
    Write $d_X:=\dim \mathbf{X}$. 
		We may use $\mathcal{M}$ to embed $B\hookrightarrow \mathbb{P}^m$ for some $m$. 
		Abusing notation slightly we continue to denote by $F$ the map induced by $F$ defined earlier which is now an endomorphism of $\mathbb{P}^n \times (\mathbb{P}^m\setminus Z)$ for a Zariski closed subset $Z\subset \mathbb{P}^m$ corresponding to the fibers above $B\setminus B_0$. 
		We also identify $\mathcal{X}^0_N, \mathcal{A}^0_N$ and their Zariski closures $\mathcal{X}_N,\mathcal{A}_N$ with their images in $(\mathbb{P}^n\times \mathbb{P}^m)^{2^{N-1}}$. 
		Letting $x=[x_0,\ldots,x_n]$ and $s=[s_0,\ldots,s_m]$ denote the projective coorinates on $\mathbb{P}^n$ and $\mathbb{P}^m$ respectively, we write
		$$F(x,s)=([f_0(x,s):\ldots:f_n(x,s)], s)$$
		for bihomogeneous $f_{i}(x,s)$ of bidegree $(d,D_1)$ in $(x,s)$ for each $i=0,\ldots,n$. 
		Since $\mathcal{X}_1\subset \mathbb{P}^n\times \mathbb{P}^m$ dominates the base but $Z$ does not, we may henceforth assume that $f_0$ is not identically zero on $(x,s)\in \mathcal{X}_1$. 
		The $\ell$-th iterate of $F$ is denoted by 
		$$F^{\ell}(x,s)=([f^{\ell}_0(x,s):\ldots:f^{\ell}_n(x,s)], s).$$
		It is easy to see that the $f^{\ell}_{i}(x,s)$ are bihomogeneous of bidegree $(d^{\ell},D_{\ell})$ in $(x,s)$ for each $i=0,\ldots,n$ and we have $D_{\ell+1}= D_1+d D_{\ell},$ so that 
		$D_{\ell}=\frac{d^{\ell}-1}{d-1}D_1 \le d^{\ell} D_1.$
		
		If we donote by 
		$$\text{pr}_{N}:(\mathbb{P}^n\times \mathbb{P}^m)^{2^{N-1}} \rightarrow \mathbb{P}^n\times \mathbb{P}^n\times  \mathbb{P}^m$$
		the projection to the first and last two factors in the product of the $2^{N}$ projective spaces above, then we have 
		\begin{align}
			\begin{split}
				\text{pr}_N^*\mathcal{O}(0,1,1) &= \mathcal{L}_N \\
				\text{pr}_N^*\mathcal{O}(0,0,1) &= \mathcal{M}_N.
			\end{split}
		\end{align}
		Since $\text{pr}_N$ has degree equal to $1$, by the projection forumla we get  
		\begin{align}\label{projection}
			(\mathcal{L}_N^{d_X}\cdot \mathcal{M}_N\cdot\mathcal{X}_N) = (\mathcal O(0,1,1)^{d_X}\cdot \mathcal O(0,0,1)\cdot \mathrm{pr}_{N*}\mathcal{X}_N).   
		\end{align} 
		We may now follow closely the arguments in \cite[Proposition 4.3]{DGH:uml}.
		Writing $[\Gamma]$ for the cycle in $\mathbb{P}^n\times \mathbb{P}^n\times\mathbb{P}^m$ representing the subvariety $\{(x,y,s)~:~y\in \mathbb{P}^n, (x,s)\in\mathcal{X}_1 \}$, we see that $\mathrm{pr}_{N*}(\mathcal{X}_N)\subset  \mathbb{P}^n\times \mathbb{P}^n\times  \mathbb{P}^m$ is an irreducible component of the intersection of $V_i:=\{y_if^{N-1}_0(x,s) -f^{N-1}_i(x,s)y_0=0 \}$ and $\Gamma$ over all $i=1,\ldots,n$.  
		In fact, as in \cite[Proposition 4.3 before equation (4.12)]{DGH:uml}, using Fulton's positivity result \cite[Corollary 12.2.(a)]{Fulton}, we get
		\begin{align}\label{cycle}
			(\mathcal{L}_N^{d_X}\cdot \mathcal{M}_N\cdot\mathcal{X}_N) \le (\mathcal O(0,1,1)^{d_X}\cdot \mathcal O(0,0,1)\cdot \mathcal{O}(d^{N-1},1,D_{N-1})\cdot [\Gamma]).
		\end{align}
		Continuing as in \cite{DGH:uml} we observe that the positive cycle $[\Gamma]\subset \mathbb{P}^n\times \mathbb{P}^n\times  \mathbb{P}^m$ is rationally equivalent with $$\displaystyle\sum_{i+p=n+m-d_X-1} a_{ip}H^{\cdot i}_1H_2^{\cdot p},$$ where $a_{ip}\in\bZ_{\ge 0}$ and  $H_1$ (respectively $H_2$) is the pullback of $\mathcal{O}(1)$ by the projection of $\mathbb{P}^n\times \mathbb{P}^n\times\mathbb{P}^m$ to the first (resp. third) factor. We thus infer that the right hand side of \eqref{cycle} is smaller than
		\begin{align}
			\sum_{i+p=n+m-d_X-1}a_{ip}(\mathcal O(0,1,1)^{d_X}\cdot \mathcal O(0,0,1)\cdot \mathcal{O}(d^{N-1},1,D_{N-1})^n\cdot \mathcal{O}(1,0,0)^{ i}\cdot \mathcal{O}(0,0,1)^{j}).
		\end{align}
		Using the linearity of intersection numbers and their vanishing properties exactly as in \cite{DGH:uml} it is then easy to infer that 
		\begin{align}
			(\mathcal{L}_N^{d_X}\cdot \mathcal{M}_N\cdot\mathcal{X}_N) \le  \sum_{\substack{i+p=n+m-d_X-1, i+i''=n \\ j'+p'=d_X,j'+j''=n \\ i''+j''+p''=n,p+p'+p''=m-1}} a_{ip} {d_X\choose
				j',p'} {n\choose i'',j'',p''} {d}^{(N-1) i''} {D_{N-1}}^{p''}. 
		\end{align}
		Noting that $p''+i''=n-j''=j'=d_X-p'\le d_X$ and that $D_{\ell} \le d^{\ell} D_1,$ we thus infer that 
		\begin{align}
			(\mathcal{L}_N^{d_X}\cdot \mathcal{M}_N\cdot\mathcal{X}_N) \le d^{(N-1)d_X} D_1^{d_X}2^{d_X}3^n \sum_{\substack{i+p=n+m-d_X-1}}
			a_{ip}. \end{align}
		The proposition then follows with 
        \begin{align}\label{effective}
        c=D_1^{d_X}2^{d_X}3^n \displaystyle\sum_{\substack{i+p=n+m-d_X-1}}
        		a_{ip}.
        \end{align}
	\end{proof}
	
    \noindent We are now ready to prove Theorem \ref{assume bog}.
    
    \smallskip 
    
	\noindent\emph{Proof of Theorem \ref{assume bog}:}
	For simplicity we write $d_X=\dim \mathbf{X}$. 
	Let $\epsilon>0$ be small. 
	By Lemma \ref{height positive}, we can find $N_{\epsilon}\in\bN$ so that 
	\begin{align*}
		(\deg_{L}(\mathbf{X})h_{\overline{L}_{\bbphi}}(\mathbf{X})+\epsilon) (d_X+1)d^{(N-1)(d_X+1)}> (\mathcal{L}^{d_X+1}_N\cdot \mathcal{X}_N)>(\deg_{L}(\mathbf{X})h_{\overline{L}_{\bbphi}}(\mathbf{X})-\epsilon) (d_X+1)d^{(N-1)(d_X+1)},
	\end{align*}
	for all $N\ge N_{\epsilon}$. 
	Let $N\ge N_{\epsilon}$ be arbitrary. 
	For simplicity we write  $\tilde{\mathcal{L}}_N:=\mathcal{L}_N|_{\mathcal{X}_N}$ and $\tilde{\mathcal{M}}_N:=\mathcal{M}_N|_{\mathcal{X}_N}$. 
	Note that $\mathcal{X}_{N+1}$ is a projective variety of dimension $d_X+1$ and its line bundles $\tilde{\mathcal{M}}_{N+1}$ and $\tilde{\mathcal{L}}_{N+1}$ are nef.
	Let $c>0$ be as in Lemma \ref{upper bound} and $y_{\epsilon}\in \mathbb{Q}_{>0}$ be such that 
   \begin{align}\label{y}
		0<\frac{(\deg_{L}(\mathbf{X})h_{\overline{L}_\bbphi}(\mathbf{X})-2\epsilon)}{c }	<y_{\epsilon}< \frac{(\deg_{L}(\mathbf{X})h_{\overline{L}_\bbphi}(\mathbf{X})-\epsilon)}{c }. 
	\end{align}
	Because of the bounds on the intersection numbers in Lemmata \ref{height positive} and \ref{upper bound}, Siu's theorem \cite[Theorem 2.2.15]{Lazarsfeld:Positivity:I} yields that the line bundle $\tilde{\mathcal{L}}_{N}-y_{\epsilon}d^{N-1}\tilde{\mathcal{M}}_{N}$ is big.  
	In particular, there is $c_1(N,\epsilon)>0$ so that 
	$$h_{\tilde{\mathcal{L}}_{N}}(Q)>y_{\epsilon}d^{N-1}h_{\tilde{\mathcal{M}}_{N}}(Q)-c_1(N,\epsilon),$$ 
	for all algebraic points $Q$ in a Zariski open subset of $\mathcal{X}_{N}$. 
	Our considerations for the heights in \eqref{height} then yield that we can find a constant $c_2(N,\epsilon)$ with
	\begin{align}\label{bySiu}
		h_{\mathfrak{L}_t}(\Phi^{N-1}_t(z))+ h_{\mathcal{M}}(t)>y_{\epsilon}d^{N-1}h_{\mathcal{M}}(t)-c_2(N,\epsilon),
	\end{align}
	for all $(z,t)$ in a Zariski open subset $\mathcal{U}_N$ of $\mathcal{X}$.
	By  \cite[Theorem 3.1]{Call:Silverman} we also have 
	\begin{align}\label{byCS}
		\hat{h}_{\Phi_t,\mathfrak{L}_t}(\Phi^{N-1}_t(z))- h_{\mathfrak{L}_t}(\Phi^{N-1}_t(z))\ge -\kappa (h_{\mathcal{M}}(t)+1),
	\end{align}
	for all $(z,t)$ in a possibly smaller open $\mathcal{U}_N\subset\mathcal{X}$ and for positive constant $\kappa$ which can can be chosen independently of $N$. 
	
	We now proceed to fix  $N=N'_{\epsilon}\ge N_{\epsilon}$ such that $\frac{\kappa+1}{d^{N-1}}<\frac{\epsilon}{c}$. 
	Combining the inequalities \eqref{bySiu} and \eqref{byCS} we get that for all $(z,t)$ in a (fixed) Zariski open subset of $\mathcal{X}$ the following holds
	\begin{align}
		\begin{split}
			d^{N-1}\hat{h}_{\Phi_t,\mathfrak{L}_t}(z)&=\hat{h}_{\Phi_t,\mathfrak{L}_t}(\Phi^{N-1}_t(z))\ge h_{\mathfrak{L}_t}(\Phi^{N-1}_t(z))-\kappa (h_{\mathcal{M}}(t)+1)\\
			&\ge (y_{\epsilon}d^{N-1}-1 -\kappa )h_{\mathcal{M}}(t)-c_2-\kappa.
		\end{split}
	\end{align}
	Dividing by $d^{N}$ and noting our choice for $N$ and $y_{\epsilon}$ and that $h_{\mathcal{M}}$ is non-negative, we get 
	\begin{align*}
		\hat{h}_{\Phi_t,\mathfrak{L}_t}(z)&\ge \left (y_{\epsilon}- \frac{\kappa+1}{d^{N-1}}\right)h_{\mathcal{M}}(t)-\frac{c_2+\kappa}{d^{N-1}}\\
		&\ge\left( \frac{\deg_{L}(\mathbf{X})h_{\overline{L}_\bbphi}(\mathbf{X})-3\epsilon}{c }\right)h_{\mathcal{M}}(t)- \frac{c_2+\kappa}{d^{N-1}}.
	\end{align*}
	The theorem follows. 
	\qed


\end{document}